\documentclass[12pt]{amsart}
\usepackage{amsmath, amsthm, amssymb}
\usepackage{fullpage}
\usepackage{lmodern,anyfontsize,fix-cm}
\usepackage{hyperref}
\usepackage{enumitem}
\usepackage{xcolor}
\usepackage[normalem]{ulem}
\usepackage[makeroom]{cancel}
\usepackage[bigdelims,vvarbb]{newpxmath} % bb from STIX
\usepackage[bottom=.96in,top=.96in,left=1.25in,right=1.25in]{geometry}
\usepackage{savesym}
\let\savedegree\bigtimes
\let\bigtimes\relax
\usepackage{mathabx}
\let\bigtimes\savedegree
\usepackage{mathtools}
\mathtoolsset{showonlyrefs}
\usepackage[backgroundcolor=gray!60,linecolor=black, tickmarkheight=.5em,textwidth=0.6\marginparwidth ]{todonotes}

\newcommand{\Hilbert}{H}
\newcommand{\Projection}{P}

\newcommand{\la}{\langle}
\newcommand{\ra}{\rangle}
\newcommand{\sgn}{\operatorname{sgn}}
\newcommand{\bds}{\boldsymbol}
\newcommand{\PV}{\operatorname{PV}}

\allowdisplaybreaks

\newcommand{\mD}{{\mathcal{D}}}

\newcommand{\qH}{{\tilde H}}
\newcommand{\tu}{{\tilde u}}
\newcommand{\hw}{{\hat w}}
\newcommand{\hr}{{\hat r}}

\newcommand{\balpha}{{\boldsymbol \alpha}}

\newcommand{\IGNORE}[1]{}
\newcommand{\ignore}[1]{}

\newcommand{\mbb}[1]{\mathbb{#1}}
\newcommand{\mb}[1]{\mathbf{#1}}

\newcommand{\bfj}{{\boldsymbol{j}}}
\newcommand{\bfk}{{\boldsymbol{k}}}

\newtheorem{theorem}{Theorem}
\newtheorem{lemma}{Lemma}

\newtheorem{definition}[lemma]{Definition}
\newtheorem{remark}[lemma]{Remark}

\numberwithin{lemma}{section}

\newcommand{\E}{E}

\numberwithin{equation}{section}

\newcommand{\R}{{\mathbb R}}

\newcommand{\Z}{{\mathbb Z}}
\newcommand{\T}{{\mathbb T}}

\renewcommand{\H}{{\mathcal H }}

\newcommand{\W}{{\mathbf W}}

\newcommand{\lin}{\text{lin}}

\begin{document}

\title{Local well-posedness of spatially quasiperiodic gravity water waves in two dimensions}

\author{Mihaela Ifrim}
\address{Department of Mathematics, 
  University of Wisconsin–Madison, 
  Madison, WI 53706, USA.}
  \email{ifrim@wisc.edu}

\author{Jon Wilkening}
\address{ Department of Mathematics, 
  University of California, 
  Berkeley, CA 94720, USA.}
  \email{wilkening@berkeley.edu}

\author{Xinyu Zhao}
\address{
  Department of Mathematical Sciences, 
  New Jersey Institute of Technology, 
  Newark, NJ 07102, USA.
  }
  \email{xinyu.zhao@njit.edu}
%%%%%%%%%%%%%%%%%%%%%%%%%%%%%%%%%%%%%%%%%%%%%%%%%%%%%%%%%%%%%

%%%%%%%%%%%%%%%%%%%%%%%%%%%%%%%%%%%%%%%%%%%%%%%%%%%%%%%%%%%

\begin{abstract}
We provide the first proof of local well-posedness for the
two-dimension\-al gravity water wave equations with spatially
quasi-periodic initial conditions. We represent the solution using
holomorphic coordinates, which are equivalent to a conformal mapping
formulation of the equations of motion. This allows us to compute the
Dirichlet-Neumann operator via the Hilbert transform, which has a
simple form in the spatially quasiperiodic setting.  We use a 
Littlewood–Paley decomposition  adapted to  the quasiperiodic  setting  and
establish multiplicative and commutator estimates in this
framework. The key step of the proof is the derivation of quasilinear
energy estimates for the linearized water wave equations with
quasiperiodic initial data.
\end{abstract}

\maketitle

\setcounter{tocdepth}{1}
\tableofcontents
%%%%%%%%%%%%%%%%%%%%%%%%%%%%%%%%%%%%%%%%%%%%%%%%%%%%%%%%%%%

\section{Introduction}

We consider the two–dimensional water wave problem of infinite depth with gravity with spatially quasi-periodic initial data. The time evolution is governed by the incompressible Euler equations with a free boundary. Under the assumption of irrotational flow, the system reduces to a one–dimensional time evolution of the position of the free surface coupled to the trace of the velocity potential on this moving interface. This nonlocal formulation in which a Dirichlet-Neumann map obviates the need to explicitly solve partial differential equations throughout the bulk fluid has been the starting point of a vast literature addressing questions such as local and global well-posedness, low-regularity theory, long-time behavior, singularity formation, wave breaking, and the construction of special classes of solutions such as traveling and standing waves. There is an extensive body of work on these topics, but we mention here only those contributions most closely related to the present goal of establishing local well-posedness of spatially quasi-periodic water waves.

A major conceptual advance in the analysis of water waves was achieved through the quasilinear system formulation introduced by Hunter, Ifrim, and Tataru \cite{HIT}, where the equations are recast using a nonlinear change of variables that reveals their underlying dispersive and transport structure. This is achieved using holomorphic coordinates via a conformal map.
This idea was first introduced in \cite{ov}, and later in \cite{zakharov2, zakharov} using a different setup.
However, the formulation in \cite{HIT} is different from other works and an alternative derivation for the evolution equations is given there.  It also expresses the normal derivative of the pressure on the boundary in 
a clean manner. This framework allows one to treat the problem as a genuinely quasilinear evolution, enabling robust energy estimates and refined structural analysis even at low regularity. Subsequent developments along these lines have led to significant progress on long-time existence, modified scattering, and the understanding of nonlinear resonances, and now form a central paradigm in the modern theory of water waves \cite{HIT, IT-global,MR3667289, AIT, MR4483135}.

In parallel, a different line of investigation has focused on the existence of special solutions such as standing waves \cite{schwartz1981semi,vandenBroeck:84,mercer:92,plotnikov01,iooss05,alazard15}, traveling-standing waves \cite{waterTS}, and more general time-quasiperiodic water waves \cite{berti2016quasi, Berti:2020:standing, Berti:2021:trav, feola2024trav}. More recently, the water wave problem has  been formulated in the spatially quasiperiodic setting and numerical methods have been developed to compute quasiperiodic traveling gravity-capillary waves \cite{quasi:trav} and pure gravity waves \cite{quasi:bif} as well as solutions of the initial value problem in infinite depth \cite{quasi:ivp} and finite depth \cite{quasi:finite}, including solutions with non-repeating patterns of overturning waves. These studies provide compelling numerical support for the existence of spatially quasiperiodic dynamics in the water wave system. However, to the best of our knowledge, no prior work has provided a proof of existence of such solutions in the fully nonlinear setting.

The purpose of the present work is to initiate a rigorous analytical study of the existence of quasiperiodic solutions for the two-dimensional gravity water wave problem in infinite depth. In subsequent work, we also plan to address the gravity–capillary water wave system. Our approach is inspired by the quasilinear perspective developed in \cite{HIT}, in the sense that we rely on a formulation that isolates the leading-order transport and dispersive mechanisms and allows for a precise control of nonlinear interactions. At the same time, the quasiperiodic nature of the solutions introduces substantial new difficulties; in particular, they possess neither translation symmetry nor decay at infinity. 
Overcoming these obstacles requires combining ideas from quasilinear dispersive analysis with tools adapted to quasiperiodic structures. These constitute the main contribution of the present work.

\subsection{Prior work.} The study of spatially quasiperiodic solutions of nonlinear dispersive PDE's has remained limited in comparison to solutions that decay at infinity or evolve on a periodic domain. In the quasiperiodic setting, solutions depend on the spatial variable through finitely many incommensurate frequencies, which leads to substantial analytical difficulties, including the lack of spatial compactness, the presence of small divisors in asymptotic expansions or Fourier representations, and the absence of a canonical Fourier basis adapted to the nonlinear dynamics.

In oceanography, modulational instabilities are believed to contribute
to the formation of rogue waves \cite{osborne2000,
  janssen2003nonlinear, onorato2006modulational,
  ablowitz2015interacting}. The nonlinear Schr\"odinger equation is
typically used to approximate the nonlinear dynamics
\cite{benney:newell:67, zak} and the growth of unstable modes is
governed by the Benjamin-Feir instability \cite{benj:feir:67}. However, for
larger-amplitude waves, weakly nonlinear theory is not an accurate
water wave model. The spectral stability of large-amplitude Stokes
waves to subharmonic perturbations has been studied by Longuet-Higgins
\cite{longuet:78}, McLean \cite{mclean:82}, MacKay and Saffman
\cite{mackay:86}, Deconinck and Oliveras \cite{oliveras:11}, Deconinck
et.~al.~\cite{trichtchenko:16}, Dyachenko and Semenova
\cite{Dyachenko:2023:quasi}, and many others. Unstable subharmonic
perturbations grow exponentially until nonlinear effects become
important, leading to interesting long-time dynamics such as
Fermi-Pasta-Ulam recurrence \cite{osborne2000, berman:FPU:2005,
  bryant:stiassnie:94}. Correctly modeling the nonlinear interactions
of component waves with incommensurate periods requires a formulation
of the fully nonlinear equations of motion in a spatially
quasi-periodic framework. This framework was recently developed in
\cite{quasi:trav, quasi:ivp, quasi:bif, quasi:finite}, where numerical
algorithms were developed to study quasi-periodic traveling waves and
solutions of the initial value problem.  Our objective here is to
develop the analytical tools needed to make this approach to the
initial value problem rigorous.

Rigorous analytical results on the well-posedness of dispersive PDEs
with spatially quasi-periodic initial data have been carried out for
the nonlinear Schr\"odinger equation \cite{Oh15:NLS, Damanik24:NLS1,
  Damanik24:NLS2, Xu25:NLS}, the Korteweg--De Vries equation
\cite{Damnik16:kdv, Tsugawa12:kdv}, the Benjamin--Ono equation
\cite{Aitzhan24:BO, PapenburgBO1, PapenburgBO2}, and the
incompressible Euler equations (without a free boundary)
\cite{Sun23:Euler}. Two main approaches have been developed in these
works: one employs energy estimates while the other relies on a
combinatorial analysis in which only quasi-periodic initial data with
exponentially decaying Fourier modes are considered.  In this paper,
we follow the first approach by working in a suitably defined Sobolev
space that accounts for the quasi-periodicity of the initial
data. However, with the exception of \cite{Sun23:Euler}, the above
works on spatially quasiperiodic solutions concern equations whose
principal part has constant coefficients or can be reduced to such a
form. As a consequence, they do not directly address fully nonlinear
or genuinely quasilinear dispersive models in which the highest-order
terms depend nonlinearly on the solution and give rise to
variable-coefficient linearizations. Moreover, the coupling between
the wave profile and the surface velocity potential introduces
additional analytical difficulties.

For the water wave system, which is either fully nonlinear or
genuinely quasilinear depending on how it is formulated, there has
been a substantial amount  of recent work on well-posedness of the equations of
motion \cite{wu2,ambrose:03,wu,coutand:shkoller:2010,abz-c1, ad, ad1, ip, HIT, MR3667289, IPcap,
  IT-global, MR4161284, MR3625189, AIT,
  MR4483135}. However, there are currently no rigorous results
establishing the existence of spatially quasiperiodic solutions. This
gap between numerical evidence and analytical theory motivates the
present work, which aims to address the interaction between
quasilinear dispersive dynamics and spatial quasiperiodicity.

To the best of our knowledge the closest result to ours in terms of
the class of initial data considered is the work
\cite{AlazardBurqZuilyUL} of Alazard, Burq, and Zuily on the uniformly
local well-posedness theory for gravity water waves. Their framework
allows for bounded, non-decaying initial data and in particular
accommodates spatially quasiperiodic profiles. However, their results
do not imply ours, since it is not established that spatially
quasiperiodic initial data are propagated by the water-wave
flow within the same class of functions.

\medskip

\subsection{\textcolor{black}{Water waves in holomorphic coordinates on the line}}
\label{ss:ww}

The idea of using a conformal mapping formulation
of the water wave system was  introduced in~\cite{zakharov} and further developed, e.g., in
  \cite{choi1999exact, dyachenko2001dynamics, ruban:2005, milewski:10,
    dyachenko2016branch}. A major reformulation was
  introduced by Hunter, Ifrim and Tataru in \cite{HIT}, which is the
  approach we adapt to the spatially quasi-periodic setting here.
A closely related formulation was adapted to the
  spatially quasi-periodic setting in \cite{quasi:ivp}, where a
  numerical algorithm was developed that employs pseudo-spectral
  spatial discretizations and an exponential time-differencing
  approach to evolve the solutions in time.  
In this section, we establish notation and briefly recall the key points in the model derivation; full details are
given in \cite{HIT}.

The fluid domain evolves in time as the image of the complex lower
half-plane
\[
\mbb C^- = \{\zeta=\alpha+i\beta\,:\, \alpha\in\mbb R\,,\,\beta<0\}
\]
under a time-dependent conformal map $z(\zeta,t)$. The free surface is
parametrized by $Z(\alpha,t)$, where $Z$ is the restriction of $z$ to
the real line, $Z=z\vert_{\beta=0}$. Let
\begin{equation*}
  \phi(\zeta,t) = \phi^\text{phys}(z(\zeta,t),t), \qquad
  \psi(\zeta,t) = \psi^\text{phys}(z(\zeta,t),t),
\end{equation*}
denote the velocity potential and stream function in conformal
($\zeta$) or physical ($z$) variables throughout the fluid, and
let $q=\phi+i\psi$ and $q^\text{phys}=\phi^\text{phys}+i\psi^\text{phys}$
denote the complex velocity potential. Their traces on the free
surface, parametrized by $\alpha$, are written
\[
\Phi=\phi\vert_{\beta=0}, \qquad
\Psi=\phi\vert_{\beta=0}, \qquad
Q(\alpha,t)=\Phi(\alpha,t)+i\Psi(\alpha,t).
\]
We also define the arclength element of the parametrization, $s_\alpha
= |Z_\alpha(\alpha,t)|$.  The normal velocity $U$ of the free surface must
match that of the fluid,
\begin{equation}\label{eq:U:Psi}
  U = (X_t,Y_t)\cdot\mb n =
  \big(\phi^\text{phys}_x,\phi^\text{phys}_y\big)\cdot\mb n =
  \Re\Big\{q^\text{phys}_z\frac{iZ_\alpha}{s_\alpha}\Big\} =
  -\frac{\Psi_\alpha}{s_\alpha},
\end{equation}
where $\mb{n} = (-Y_\alpha,X_\alpha)/s_\alpha$. 
On the other hand, the tangential velocity $V$  of the free surface parameterization is chosen to maintain
a conformal mapping parametrization.
 
The key observation, first noted by Zakharov \emph{et~al.} in \cite{zakharov}, is that
$z_t/z_\zeta$ is bounded and analytic in the lower half-plane,  (see also
  \cite{quasi:ivp} for sufficient conditions in the quasi-periodic
  setting), and its imaginary part agrees with $U/s_\alpha$ when
$\beta=0$:
\begin{equation}\label{eq:zt:za}
  \frac{z_t}{z_\zeta}\bigg\vert_{\beta=0} = \frac{Z_t}{Z_\alpha} = \frac{(X_\alpha X_t+Y_\alpha Y_t)+i(X_\alpha Y_t
      - Y_\alpha X_t)}{s_\alpha^2} = \frac{V + iU}{s_\alpha}.
\end{equation}
Thus, maintaining $Z=z\vert_{\beta=0}$ requires $V/s_\alpha = H[U/s_\alpha]+C_1(t)$,
where $C_1(t)$ is an arbitrary constant (independent of $\alpha$) at each fixed time
and $H$ is the Hilbert transform, defined via the Cauchy principal value integral
\begin{equation}\label{eq:Hf:PV}
  \Hilbert[u](\alpha) = \frac{1}{\pi}\PV
  \int_{-\infty}^\infty\frac{u(\xi)}{\alpha-\xi}\,d\xi.
\end{equation}
In the present work, following \cite{HIT}, we set $C_1(t)=0$, which is also
the most natural choice in infinite depth in the quasiperiodic setting, as explained in
\cite{quasi:ivp}. Any other choice, say $C_1(t)=-\alpha_0'(t)$, leads
to the same physical solution, but reparametrized via
$\alpha\to\alpha-\alpha_0(t)$ (see also \cite{quasi:trav} for the quasiperiodic  setting). Solving \eqref{eq:zt:za}
gives
\begin{equation}
  \begin{pmatrix} X_t \\ Y_t \end{pmatrix} = \begin{pmatrix} X_\alpha & -Y_\alpha \\
    Y_\alpha & X_\alpha \end{pmatrix}\begin{pmatrix} V/s_\alpha \\ U/s_\alpha \end{pmatrix}.
\end{equation}
By \eqref{eq:U:Psi}, this is equivalent to $Z_t + FZ_\alpha=0$, where, following \cite{HIT}, we introduce the auxiliary variables
\begin{equation}\label{eq:F:P:def}
  F = P\left[ \frac{Q_\alpha - \bar Q_\alpha}{J}\right] , \qquad J = s_\alpha^2 = |Z_\alpha|^2, \qquad
  P = \frac12(I-iH).
\end{equation}
We note that $F=-(V+iU)/s_\alpha$, and $V$ is the tangential velocity of
the curve parametrization rather than of the fluid. The evolution equation
for the surface velocity potential,
\begin{equation}\label{eq:Phi:t}
  \Phi_t - \big(\phi_x^\text{phys},\phi_y^\text{phys}\big)\cdot (X_t,Y_t) +
  \frac{\Phi_\alpha^2+\Psi_\alpha^2}{2J} + Y = 0,
\end{equation}
is obtained from the Bernoulli
equation $\phi^\text{phys}_t +
\frac12|\nabla\phi^\text{phys}|^2+\frac{p}{\rho}+gy=C_2(t)$, where
$p$ is the pressure, $\rho$ is the fluid density, $g$ is the
acceleration of gravity (which has been set to 1 in \eqref{eq:Phi:t}
  after non-dimensionalization), and $C_2(t)$ is another arbitrary
constant that is allowed to vary in time. Neglecting the effects
of surface tension causes $p/\rho$ to be constant on the free
surface, which is eliminated by setting $C_2=p/\rho$.
Applying $2P$ to \eqref{eq:Phi:t} converts $\Phi_t$ into $Q_t$,
$\Re\{q^\text{phys}_z Z_t\}$ into $(Q_\alpha/Z_\alpha)Z_t$, and $Y$
into $-i(Z-\alpha)$, where we used $\Psi=-H\Phi$ and
$HY=X-\alpha$. The latter condition assumes $z(\zeta,t)-\zeta$ is
bounded in $\mbb C^-$ and that $P_0(X-\alpha)=0$ initially, where
$P_0$ projects a quasi-periodic function onto its mean.  Our choice of
$C_1(t)=0$ above preserves this mean value in infinite depth
\cite{quasi:ivp}, though this is not true in
finite depth \cite{quasi:finite}.  Together with the kinematic condition, we obtain
\begin{equation}\label{hww}
\left\{
\begin{aligned}
& \raisebox{-6pt}{$Z_t + F Z_\alpha = 0,$} \\[-3pt]
& Q_t + F Q_\alpha -i (Z-\alpha) + P\left[ \frac{|Q_\alpha|^2}{J}\right] = 0,
\end{aligned}
\right.
\end{equation}
which are the water wave equations in holomorphic coordinates from \cite{HIT}.

As $P$ retains exactly half of the zero mode, it is not a projector
when acting on periodic or quasi-periodic functions, which generally
have a non-vanishing zero mode. For a detailed discussion of the
structure of the operator $P$, we refer the reader to the work of the
first author with Hunter and Tataru~\cite{HIT}. In the spatially
periodic and
quasi-periodic settings, there is a projection $P_0$ onto this zero
mode (see \S\ref{ss:quasi} below), and we define
\begin{equation}
\label{def:PsharpPiPr}
P^\sharp := P - \frac12 P_0, \qquad \bar P := \frac12(I+iH), \qquad
\bar P^\sharp := \bar P - \frac12 P_0
\end{equation}
so that
\begin{equation}
I = P+\bar P = P^\sharp + \bar P^\sharp + P_0, \qquad
(P^\sharp)^2=P^\sharp, \qquad (\bar P^\sharp)^2=\bar P^\sharp, \qquad
P^\sharp\bar P^\sharp = 
\bar P^\sharp P^\sharp = 0
\end{equation}
and $P_0P^\sharp = P^\sharp P_0 = P_0\bar P^\sharp = \bar P^\sharp P_0
= 0$.   We also define the related projectors 
\begin{equation}
%\label{def:PsharpPiPr}
 P^r := P^\sharp + \Re P_0, \qquad P^i := P^\sharp + i \Im P_0,
\end{equation} 
and similarly $\bar P^r$, $\bar P^i$, so that we have the following relations:
\begin{equation}
\label{eq:relPsharpPiPr}
 I = P^i + \bar P^r = P^r + \bar P^i, \qquad 
P^i \bar P^r = P^r \bar P^i = 0, \qquad P^i = -i P^r i.
\end{equation} 
Both $Z-\alpha$ and $Q$ will evolve in state spaces that are closed subspaces of holomorphic functions, satisfying $\bar
P^r(Z-\alpha)=0$ and $\bar P^i Q=0$,  within various Sobolev
spaces.
The Fourier transform of such functions (which are tempered
  distributions in the spatially quasi-periodic case) are supported on
$(-\infty,0]$ and admit bounded holomorphic extensions to $\mbb C^-$.

As the system \eqref{hww} is fully nonlinear, a standard procedure is
to convert it into a quasilinear system by differentiating it. To
simplify the result, we introduce the auxiliary real function $b$,
which we call the {\em advection velocity}, given by
\begin{equation}
  b = \Re F =  P \left[\frac{{Q}_\alpha}{J}\right] +  \bar P\left[\frac{\bar{Q}_\alpha}{J}\right]
  = \frac{V_f-V}{s_\alpha},
\label{defb1}
\end{equation}
where $V_f=\Phi_\alpha/s_\alpha$ is the tangential component of the fluid velocity
at the free surface while $V_f-V$ measures this tangential velocity relative to that
of the curve parametrization. 

Introducing
\begin{equation}
  W(\alpha,t) = Z(\alpha,t) - \alpha,
\end{equation}
and using $b$, the system \eqref{hww} may be written
\begin{equation}
\left\{
\begin{aligned}
& W_t + b (1+W_\alpha) = \frac{\bar Q_\alpha}{1+\bar W_\alpha}, \\
&Q_t +  b Q_\alpha - iW =  \bar P\left[ \frac{|Q_\alpha|^2}{J}\right],
\end{aligned}
\right.
\end{equation}
where the terms on the right are antiholomorphic and disappear when
the equations are projected onto the holomorphic
space via $P^\sharp$. Differentiating with respect to $\alpha$
and following identical arguments to \cite{HIT} yields the
system
\begin{equation} \label{ww2d-diff-real}
\left\{
\begin{aligned}
 & \W_{ t} + b \W_{ \alpha} + \frac{(1+\W) R_\alpha}{1+\bar \W}   =  (1+\W)M,
\\
& R_t + bR_\alpha  =  i\left(\frac{\W - a}{1+\W}\right),
\end{aligned}
\right.
\end{equation}
where
\begin{gather}
  \W:=W_\alpha, \qquad R := \frac{Q_\alpha}{1+W_\alpha}, \qquad
  Y := \frac{\W}{1+\W}, \\
  \label{defa}
  a := i\left(\bar P \left[\bar{R} R_\alpha\right]- P\left[R\bar{R}_\alpha\right]\right), \\
  M :=  \frac{R_\alpha}{1+\bar \W}  + \frac{\bar R_\alpha}{1+ \W} -  b_\alpha =
  \bar P [\bar R Y_\alpha- R_\alpha \bar Y]  + P[R \bar Y_\alpha - \bar R_\alpha Y].
\end{gather}
Here $R(\alpha,t)$ parametrizes the (complex conjugate of the) fluid velocity on the free surface
and $a$ is the real {\em frequency-shift}. The auxiliary functions $M$ and $Y$ were introduced
to simplify the formulas and avoid rational expressions in many places above and in the sequel.
We note that $b$ in \eqref{defb1} can also be written
\begin{equation}
  b = 2\Re P[R(1-\bar Y)].
\end{equation}
The system \eqref{ww2d-diff-real} governs an evolution in the space of holomorphic functions, and will be used both directly and in its projected version.
The functions $a$ and $b$ also play a fundamental role in the linearized equations, which are computed in Section~\ref{s:linearized}.

\subsection{The set-up for quasiperiodic solutions} \label{ss:quasi}
We consider the class of  quasi-periodic functions $\tu(\alpha)$ of the form 
\begin{equation}\label{general_quasi_form}
  \tu(\alpha) =  u(\bds{k} \alpha), \qquad
  u(\bds\alpha) = \sum_{\boldsymbol{j}\in\mathbb{Z}^d}\hat{u}_{\boldsymbol{j}}
  e^{i\la\boldsymbol{j},\,\boldsymbol{\alpha}\ra}, \qquad
  \alpha\in\mbb R, \;\; \bds\alpha,\bds k \in \mathbb{R}^d,
\end{equation}
where $\la\cdot , \,\cdot\ra$ denotes the standard inner
product in $\mathbb{R}^d$,  and $u$ is a periodic
function defined on the $d$-dimensional torus
$\mathbb{T}^d := \mbb R^d\big/(2\pi\mbb Z)^d,$ which we will refer to as the periodic parent functions of $\tu$.

Entries of the vector $\boldsymbol{k}$ are called the basic wave
  numbers (or basic frequencies) of $\tu$ and are required to be
  linearly independent over $\mathbb{Z}$. If $\bds{k}$ is
  given, one can reconstruct the Fourier coefficients
  $\hat u_{\bds j}$ from $\tu$ or $u$ via
\begin{equation}\label{uhat:from:u}
\begin{aligned}
  \hat{u}_{\bds{j}} =& \lim_{L\to\infty} \frac{1}{2L} 
  \int_{-L}^L \tu(\alpha) e^{-i\la \bds{j}, \bds{k}\ra\alpha} d\alpha, \\
    =& \frac1{(2\pi)^d}\int_{\mbb T^d}
    u(\bds\alpha)e^{-i\la\bds
      j,\bds\alpha\ra} \, d\alpha_1\cdots\, d\alpha_d.
  \end{aligned}
  \qquad \bds{j}\in\mathbb{Z}^d.
\end{equation}

We define the projection operator 
$\Projection_0$ that acts on $\tu$ and $u$ via
\begin{equation}\label{eq:proj0}
  \Projection_0 [\tu] = \Projection_0 [u] 
  = \hat u_{\bds 0}.
\end{equation}
Note that  $\Projection_0$ returns the mean value, viewed as a constant function on $\mathbb R$ or $\mathbb T^d$. There
  are two versions of  $P_0$, one acting on quasi-periodic
  functions defined on $\mbb R$ and one acting on torus functions
  defined on $\mbb T^d$.

 We introduce a quasi-periodic
Hilbert transform  $\Hilbert$, which is defined so that it has the property $\Projection_0\Hilbert[\tu]=0$.

\begin{definition}
  The Hilbert transform of a quasi-periodic function
  $\tu(\alpha)$ of the form \eqref{general_quasi_form} is defined to be
  \begin{equation}\label{eq:hilb:def}
    \Hilbert[\tu](\alpha) = \sum\limits_{\boldsymbol{j}\in\mathbb{Z}^d}
            (-i)\sgn(\la\boldsymbol{j},\,\boldsymbol{k}\ra)
            \hat{u}_{\boldsymbol{j}}
            e^{i \la\boldsymbol{j},\,\boldsymbol{k}\ra\alpha}.
  \end{equation}
\end{definition}
Here \emph{sgn} is defined to take the values $\{-1,0,1\}$.
This agrees with the standard definition \eqref{eq:Hf:PV} of the
Hilbert transform as a Cauchy principal value integral 
\cite{dyachenko2016branch}.

  As with  $P_0$, there is an analogous operator on $L^2(\mbb
    T^d)$ such that $\Hilbert[\tu](\alpha)=H[u](\bds k\alpha)$.
  The formula is
  \begin{equation}
    \Hilbert[ u](\bds\alpha) = \sum\limits_{\boldsymbol{j}\in\mathbb{Z}^d}
            (-i)\sgn(\la\boldsymbol{j},\,\boldsymbol{k}\ra)
            \hat{u}_{\boldsymbol{j}}
            e^{i \la\boldsymbol{j},\,\boldsymbol{\alpha}\ra}.
  \end{equation}

Now that we have a definition of the Hilbert transform in place, 
we can give a precise meaning in the Fourier space to the constructions introduced earlier in the introduction and,  
in particular, to the projectors \eqref{def:PsharpPiPr}.

On the torus we have the differentiation operator
\begin{equation*} \label{eq:directional:dev}
\partial_\alpha := k_1 \partial_1 + \cdots + k_d \partial_d,
\end{equation*}
which has symbol $i \xi_\alpha$ where 
\[
\xi_\alpha : = k_1 \xi_1 + \cdots k_d \xi_d.
\]
Then the formulation of the water wave equations on the 
torus $\T^d$ is identical to \eqref{ww2d-diff-real}, but where the projection $P$ is defined in \eqref{eq:F:P:def},  selecting lattice 
frequencies $\bds\xi$ with $\xi_\alpha < 0$ plus half of  the zero mode.
From now on, we focus on the periodic parent functions defined on $\T^d$ and establish the local well-posedness in suitably defined Sobolev spaces for these functions.

\subsection{Function spaces} \label{ss:function}
As the basis for the function spaces on the torus we use 
the classical Sobolev spaces $H^s$. The $H^s$ norm on the torus induces a corresponding norm on the quasiperiodic
functions on the line, defined as 
\begin{equation}
\| \tu \|_{\qH^s} := \| u \|_{H^s(\T^d)}   . 
\end{equation}

We remark that the $\qH^s$ norm on the line differs from the $H^s$ norm even locally. The easy connection is that
\[
\qH^s \subset H^{s-\frac{d-1}{2}}_{loc}, \qquad s > \frac{d-1}{2} ,
\]
which follows by Sobolev embeddings. Any converse to this, on the other hand, depends on the diophantine properties of $\bds k$.

On top of the $H^s$ spaces, our evolution also involves  the $\alpha$-direction differentiation, so we introduce a 
two-parameter family of spaces 
\[
\| u \|_{H^{s,\theta}} := \| \partial_\alpha^{\leq\theta}u \|_{H^s}.
\]
This is directly defined for nonnegative integer $\theta$, but by interpolation it extends to all nonnegative $\theta$,
\[
\| u \|_{H^{s,\theta}} := \| \langle D_\alpha\rangle^{\theta}u \|_{H^s}, \qquad  \langle D_\alpha\rangle
= (1+D_\alpha^2)^\frac12,
\]
where the symbol of $D_\alpha$ is $ \xi_\alpha$.

For the study of the differentiated gravity wave system, \eqref{ww2d-diff-real},  in the $\mathbb{T}^d$ setting, we will employ the following spaces 
\begin{equation}
\H^s := H^s(\mbb T^d) \times H^{s,\frac12}(\mbb T^d), \qquad s > \frac{d+1}{2}.
\end{equation}

\subsection{Control parameters}

For the estimate in this paper, an important role is played by our  control parameters, which are defined as 
\begin{equation}
A := \|(\W,R)\|_{\H^{s-\frac12}}, \qquad B := 
\|(\W,R)\|_{\H^{s}}.  
\end{equation}
Here $A$ is close to scaling, and will govern elliptic estimates at fixed time. $B$, on the other hand, is roughly one-half derivative above scaling and will govern the energy growth rate.

Since $s > \frac{d+1}{2}$ these in particular guarantee, by Sobolev embeddings, that 
\begin{equation}
\| \W\|_{L^\infty} + \| \langle D_\alpha\rangle ^\frac12 R\|_{L^\infty} \lesssim A,
\qquad 
\| \langle D_\alpha\rangle^\frac12 \W\|_{L^\infty} + \| \partial_\alpha R\|_{L^\infty} \lesssim B.
\end{equation}

\subsection{The main result}

In brief, this is a comprehensive local well-posedness theory for spatially quasiperidic solutions for water waves. Classically, even in a semilinear setting one would rely on: (i) very high regularity, as loss of derivatives is a major issue; (ii) and structural properties (symmetries, Hamiltonian form, null conditions).  In short this happens as the background quasiperiodic geometry interacts badly with quasilinearity nature of the equations. Instead, in this work we work with a fully\footnote{Even after differentiation, the resulting system is still fully nonlinear, as opposed to quasilinear, due to the nonlinear dependence on $\W$ in the  second equation.} nonlinear system, and we are able to consider solutions at low regularity, just one-half derivative above scaling. Our main result
for the gravity wave system in the quasiperiodic setting is as follows:

\begin{theorem}\label{thm:lwp}
The differentiated gravity wave system \eqref{ww2d-diff-real} is locally well-posed in $\H^s(\T^d)$ for $s > \frac{d+1}2$.   
\end{theorem}

By the trace theorem, the solutions in our result have merely $\H^{1+}_{loc}$
regularity on the real line.

This problem is quasilinear, so the local well-posedness should be interpreted in the Hadamard style, see \cite{IT-primer}.

\subsection{Outline of the proof}

 The paper is structured in a modular fashion, allowing for a clean presentation that decouples the various steps involved in the proof of our main result, as outlined below. Section~\ref{s:pre} introduces the quasiperiodic setting and the associated Littlewood–Paley decomposition, which will serve as an essential preparatory component for the analysis carried out in Section~\ref{s:existence} and Section~\ref{s:rough sol}.

In Section~\ref{s:linearized}, we collect the energy estimates and establish the well-posedness theory for the linearized equation. The analysis in this section relies on a carefully crafted framework which, in turn, requires a number of refined estimates. These are divided into two main categories: (i) Section~\ref{s:multi}     gathers all multilinear and commutator bounds used throughout the paper, many of which apply in a more general setting than that of water waves; whereas (ii) in Section~\ref{s:water} we focus on estimates specific to the water-wave problem. The latter are established for both the unknowns and coefficients arising from the diagonalization of the water-wave system.

As is well known, the analysis of the linearized problem constitutes the core of the arguments that follow. It provides the backbone for the energy estimates derived for the fully nonlinear system  in Section~\ref{s:ee}, and it also serves as a key ingredient in the constructive  proof of the existence of regular solutions in Section~\ref{s:existence}.
It also provides difference bounds for solutions,
which allow us to pass from rough to smooth solutions in this quasilinear setting, a step  which is carried out in Section~\ref{s:rough sol}.

%%%%%%%%%%%%%%%%%%%%%%%%%%%
%%%%%%%%%%%%%%%%%%%%%%%%%%%
%%%%%%%%%%%%%%%%%%%%%%%%%%%
%%%%%%%%%%%%%%%%%%%%%%%%%%%
\section{The Littlewood-Paley decomposition for quasiperiodic functions}
\label{s:pre}
%%%%%%%%%%%%%%%%%%%%%%%%%%%
%%%%%%%%%%%%%%%%%%%%%%%%%%%
%%%%%%%%%%%%%%%%%%%%%%%%%%%
%%%%%%%%%%%%%%%%%%%%%%%%%%%

In this section, we recall the Littlewood–Paley decomposition for decaying functions on $\R$ and then extend it to the quasi-periodic setting,
developing the basic framework needed for the \emph{paradifferential calculus}, following the approach of Bony \cite{Bony1981}. These notions are used in the local well-posedness existence scheme later in Section~\ref{s:existence}.

%%%%%%%%%%%%%%%%%%%%%%%%%%%%%
%%%%%%%%%%%%%%%%%%%%%%%%%%%%%
\subsection{Littlewood-Paley decomposition on \texorpdfstring{$\mathbb{R}$}{R}}
\label{sec:lp:decomp}
%%%%%%%%%%%%%%%%%%%%%%%%%%%%%
%%%%%%%%%%%%%%%%%%%%%%%%%%%%%

One important tool in
dealing with dispersive equations is the Littlewood-Paley
decomposition.  We recall its definition and also its usefulness in
the next paragraph. We begin with the Riesz decomposition
\[
 1 = P_- + P+,
\]
where $P_\pm$ are the Fourier projections to $\pm [0,\infty)$. From 
\[
\widehat{Hf}(\xi)=-i\sgn (\xi)\, \hat{f}(\xi),
\]
 we observe that
\begin{equation}\label{hilbert-re}
iH = P_{+} - P_{-}.
\end{equation}

Let $\chi \in C_c^\infty(\mathbb{R})$ satisfy
\[
\chi(\xi)=1 \quad \text{for } |\xi|\le 1,
\qquad
\chi(\xi)=0 \quad \text{for } |\xi|\ge 2,
\]
and define
\[
\varphi(\xi) := \chi(\xi) - \chi(2\xi).
\]
For any integer $l\geq 0$, set
\begin{equation*}\label{eq:varphi:l}
\varphi_l(\xi) := \varphi(2^{-l}\xi).
\end{equation*}

We define the Littlewood-Paley operators $P_{\leq l}$ and $P_l := P_{\leq l} - P_{\leq l-1}$ (with the convention $P_{\leq -1} = 0$) by
\begin{equation}
\widehat{P_{\leq l} f}(\xi) := \chi(2^{-l} \xi) \hat f(\xi), \qquad
\widehat{P_{l} f}(\xi) := \varphi_l(\xi) \hat f(\xi).
\end{equation}

Note that all the operators $P_l$,
$P_{\leq l}$ are bounded on all translation-invariant Banach spaces,
thanks to Minkowski's inequality.  We define $P_{>l} := P_{\geq l+1} 
:= 1 - P_{\leq k}$.

For simplicity and because $P_{\pm}$ commutes with the
Littlewood-Paley projections $P_l$ and $P_{<l}$, we will introduce the
following notation $P^{\pm}_l:=P_lP_{\pm}$ , respectively
$P^{\pm}_{<l}:=P_{\pm}P_{<l}$.  In the same spirit, we introduce the
notations $f^{+}_l:=P^{+}_l f$, and $f^{-}_l:=P^{-}_l f$,
respectively.

Given the projectors $P_l$, we also introduce additional projectors $\tilde P_l$
with slightly enlarged support (say, by $2^{l-4}$) and symbol equal to $1$ in the support of $P_l$.

From Plancherel's theorem  we have the bound
\begin{equation}\label{eq:planch}
 \| f \|_{H^s(\R)} \approx (\sum_{l=0}^\infty \| P_l f\|_{H^s(\R)}^2)^{1/2}
\approx (\sum_{l=0}^\infty 2^{2ls} \| P_l f\|_{L^2(\R)}^2)^{1/2} 
\end{equation}
for any $s \in \mathbb{ R}$.

%%%%%%%%%%%%%%%%%%%%%%%%%%%
%%%%%%%%%%%%%%%%%%%%%%%%%%%
\subsection{Littlewood--Paley decomposition in the quasiperiodic setting}
\label{ss:LP:quasi}

Recall that a function $\tu:\mathbb R\to\mathbb C$ is said to be quasiperiodic with frequency vector
$\bds k\in\mathbb R^d$ if it can be written in the form
\[
\tu(\alpha) = u(\bds k\alpha),
\]
where $u:\mathbb T^d\to\mathbb C$ is $2\pi$--periodic. When the components of $\bds k$ are
rationally independent, the map $\alpha\mapsto \bds k \alpha \pmod{2\pi}$ defines a dense linear flow on
$\mathbb T^d$. Rather than working directly with quasiperiodic functions on $\mathbb R$, we
carry out all analytic constructions on the torus $\mathbb T^d$, where Fourier analysis and
Littlewood--Paley theory are available. Results are then interpreted for quasiperiodic
functions via the composition $\tu(\alpha)=u(\bds k \alpha)$.

Accordingly, we fix a vector $\bds k\in\mathbb R^d$ with rationally independent components, and
let 
\[
u(\balpha) = \sum_{\bds j\in\mathbb{Z}^d} \hat u_{\bds j}\, e^{i\la\boldsymbol{j},\,\boldsymbol{\alpha}\ra},
\qquad \balpha \in \mathbb{T}^d,
\]
denote the Fourier series of a function on the torus.

We define the Littlewood--Paley projections adapted to the quasiperiodic direction $\bds k$ by
\begin{equation}\label{eq:LP-quasiperiodic}
P_l^{\alpha} u(\balpha)
:= \sum_{\bds j\in\mathbb{Z}^d}
\varphi_l(\la \bds j,\, \bds k  \ra)\,\hat u_{\bds j}\, e^{i \la\bds j, \, \bds \alpha \ra},
\qquad l\ge 0,
\end{equation}
and the associated low--frequency cutoff
\begin{equation}\label{eq:LP-low}
P_{\le l}^{\alpha} u(\balpha)
:= \sum_{\bds j\in\mathbb{Z}^d}
\chi(2^{-l} \la \bds j, \bds k \ra)\,\hat u_{\bds j} \, e^{i \la\bds j, \, \balpha \ra}.
\end{equation}

These operators localize Fourier modes according to the scalar frequency $\la \bds j, \, \bds k \ra$, rather
than the full lattice norm $|\bds j|$. In particular, they are adapted to operators whose symbols
depend only on $\la \bds j, \, \bds k \ra$, as is typical in quasiperiodic problems.

By Plancherel's theorem on $\mathbb T^d$, the Littlewood--Paley projections defined in
\eqref{eq:LP-quasiperiodic} yield the norm equivalence
\begin{equation}\label{eq:LP-equivalence}
 \| u \|_{H^{0,s}(\mathbb T^d)}
 \approx
 \Big( \sum_{j=0}^\infty 2^{2ls}\, \| P_l^{\alpha} u \|_{L^2(\mathbb T^d)}^2 \Big)^{1/2}
\end{equation}
for any $s\in\mathbb R$. We will sometimes omit $\alpha$ in the projector and use the following notation for simplicity:
\begin{equation*}
P_\lambda u = u_\lambda = P_l^\alpha u, \qquad \lambda = 2^l \geq  1.
\end{equation*}

\section{Energy estimates and local well-posedness for
the linearized equation}
\label{s:linearized}
In this section, we derive the linearized system associated with the evolution equation \eqref{ww2d-diff-real}. As is by now well understood, the linearized system plays a crucial role in establishing well-posedness. We therefore devote the remainder of this section to a detailed analysis of its structure and properties, focusing in particular on the energy estimates, which form the second main objective of this section.

\subsection{The linearized system}
We begin by recalling the differentiated system in \eqref{ww2d-diff-real}:
\begin{equation} \label{eq:W:R}
\left\{
\begin{aligned}
 & \W_{ t} + b \W_{ \alpha} + \frac{(1+\W) R_\alpha}{1+\bar \W}   =  (1+\W)M,
\\
& R_t + bR_\alpha = i\left(\frac{\W - a}{1+\W}\right).
\end{aligned}
\right.
\end{equation}

We denote the linearized variables around a solution $(\W, R)$ by $(w, r)$, which are restricted to the class of holomorphic functions. We define
\begin{equation}
D_t := \partial_t + b\partial_\alpha.
\end{equation}

The linearization of the differentiated water wave system \eqref{eq:W:R} reads
\begin{equation}\label{lin(hwhr)}
\left\{
\begin{aligned}
 &D_t w + (1 - \bar Y)(1 + \W_\alpha )  r_\alpha 
 = f, \\
&D_t r  -i(1 + a)(1 - Y)^2 w = g,
\end{aligned}
\right.
\end{equation}
where the source terms $f$ and $g$ are given by 
\begin{equation}\label{e:source}
\left\{
\begin{aligned}
&f =  - (1-\bar Y)R_\alpha w + (1-\bar Y)^2(1+\W)R_\alpha \bar w + M w -\W_\alpha \delta b + (1+\W)\delta M, \\
&g = - R_\alpha \delta b - i(1-Y)\delta a, 
\end{aligned}
\right.
\end{equation}
with $\delta b$, $\delta a$ and $\delta M$ 
arising as the linearizations of $b,a$ and $M$,
given by
\[
\left\{
\begin{aligned}
&\delta b = 2\Re P[(1 - \bar Y) r - (1 - \bar Y)^2 R \bar{w}], \\
&\delta a =  2 \Im P[\bar R_\alpha r + R\bar r_\alpha], \\
&\delta M = 2\Re P[r \bar Y_\alpha -2 R(1 - \bar Y) \bar{w} \bar Y_\alpha + R (1 - \bar Y)^2 \bar{ w}_\alpha - \bar{r}_\alpha Y - \bar R_\alpha (1 - Y)^2 w] .
\end{aligned}
\right.
\]

\subsection{Energy estimates for the linearized system}
We note that a conserved energy for the linearized equation around the zero solution is 
\begin{equation}\label{E_{0}}
\E_{0} (w,r) := \int_{\T^d} \frac12 |w|^2 + \frac{1}{2i} (r \bar r_\alpha - 
\bar r r_\alpha) +\vert r\vert^2 \,\, d\balpha.
\end{equation}
Around a nonzero solution we  instead define the linearized energy as 
\begin{equation}
\label{eq:used-energy}
E_{\lin}(w,r) := \int_{\T^d} (1+a) |1-Y|^2 |w|^2 + \Im ( r_\alpha \bar r)   +\vert r\vert ^2 \, d\balpha. 
\end{equation}
Here we note the quasilinear type correction factor $(1+a) |1-Y|^2$, which arises due to 
the balance of the coefficients on the left in 
\eqref{lin(hwhr)}.

\begin{remark}
The expression in \eqref{eq:used-energy} follows from the work of Ai-Ifrim-Tataru in \cite{AIT}; there the analysis uses two linearizations which end up being connected via a nonlinear transformation:
\begin{itemize}
\item one linearization, denoted by $(w, r)$,
around the solution of the undifferentiated system for $(W, Q)$, and  
\item  the second one, denoted by $(\hw,\hr)$, and associated to the $(\W, R)$-differentiated system.
\end{itemize} 
Throughout this paper, we omit the hat notation used in \cite{AIT}.
\end{remark}

With the definition above, we can now state our main energy estimate for the linearized equation:

\begin{theorem}
\label{plin-short}
The energy functional $E_{\lin}(w,r)$ satisfies
\begin{enumerate}
    \item Coercivity 
\begin{equation}
\label{coercivity-lin}
 E_{\lin}(w,r) \approx \| (w,r)\|_{\H^0}^2   ;
\end{equation}
\item Growth bound 
\begin{equation}
\label{ee-lin}
   \left|\frac{d}{dt}  E_{\lin}(w,r)\right| \lesssim_A B \| (w, r)\|_{\H^0}^2 .
\end{equation}
\end{enumerate}
\end{theorem}

One consequence of this result is the  $\H^0$ well-posedness for the linearized flow,
whenever the control parameters $A$ and $B$ remain bounded.

Here we remark that $w$ and $r$ are restricted to the class of holomorphic functions, which further have zero average.
Because of this, we can harmlessly insert projections both on the left and on the right of (\ref{lin(hwhr)}).
The projected equations are equivalent to the full equations, so this makes no difference 
for the above result. However, it becomes important when we try to treat the RHS 
perturbatively and think of  the linearized  equation (\ref{lin(hwhr)}) as a perturbation of the following evolution
\begin{equation}\label{lin(hwhr)usf}
\left\{
\begin{aligned}
 & \mD_t w + P^\sharp[(1 - \bar Y)(1 + W_\alpha) r_\alpha ]
 = P^\sharp[f],
\\
&\mD_t r   -iP^\sharp[(1 + a)(1 - Y)^2w] = P^\sharp[g],
\end{aligned}
\right.
\end{equation}
where $\mD :=P^\sharp[D_t]$. We will refer to this system as the reduced linearized equation. Here the projectors ensure that the solutions remain in the space of holomorphic functions with zero mean.
The full linearized equation \eqref{lin(hwhr)} corresponds to taking the source terms $(f,g)$ 
as in \eqref{e:source}.

\begin{remark}
The use of the projectors $P^\sharp$  in the above projections is justified by the fact that for the full linearized equations we seek to view the flow as an evolution in the space of 
zero-average  holomorphic functions. We note 
that this would no longer be possible at the level of the undifferentiated equations, where, 
as explained in \cite{HIT}[Appendix A], one needs to use the projectors $P^i$, respectively $P^r$ in the two equations.
\end{remark}

With this setup, Theorem~\ref{plin-short} will be 
interpreted as a consequence of the theorem below for the reduced linearized equation:

\begin{theorem}
\label{plin-reduced}
There exists an energy functional $E_{lin}(w,r)$ for the reduced linearized equation, so that the followings hold

\begin{enumerate}
    \item Coercivity
\begin{equation}
\label{coercivity}
 E_{lin}(w,r) \approx \| (w,r)\|_{\H^0}^2;
\end{equation}
\item Growth bound
\begin{equation}
\label{ee}
   \left|\frac{d}{dt}  E_{lin}(w,r)\right| \lesssim_A B \| (w,r)\|_{\H^0}^2 + \| (P^\sharp f,P^\sharp g)\|_{\H^0}.
\end{equation}
\end{enumerate}
\end{theorem}

In principle, this should also imply $\H^0$ well-posedness for the reduced linearized flow. Proving this would require interpreting the adjoint reduced linearized equation as a perturbation of the reduced linearized equation,
which therefore satisfies similar bounds.

\begin{proof}[Proof of Theorem~\ref{plin-reduced}]
We compute the time derivative of the linearized energy $E_{lin}$ 
for solutions to \eqref{lin(hwhr)usf}
with $( f, g) = 0 $. For the $w$ term we 
substitute the projected material derivative 
with the full material derivative modulo a commutator term, obtaining
\begin{equation}\label{dt-w}
\begin{aligned}
\frac{d}{dt}  \int_{\mathbb{T}^d} (1+a) |1-Y|^2 |w|^2   \, d\balpha = & \ 2 \Re \int_{\T^d} \bar w (1+a)|1-Y|^2
[b,P^\sharp] \partial_\alpha w \, d \balpha
\\ & + \int_{\T^d}  b_\alpha (1+a) |1-Y|^2 |w|^2   \, d\balpha 
\\ &  +\int_{\T^d} D_t ((1+a)|1-Y|^2) |w|^2\, d\balpha
\\ & - 2 \Re \int_{\T^d}(1+a) |1-Y|^2 \bar w P^\sharp[(1-\bar Y)(1+\W) r_\alpha] \, d\balpha
\\ & +  2 \Re \int_{\T^d}(1+a) |1-Y|^2 \bar w P^\sharp f \, d\balpha .
\end{aligned}
\end{equation}
For both $r$ terms we also 
substitute the projected material derivative 
with the full material derivative modulo a commutator term, but this time we can freely drop the projection so there is no commutator term,
\begin{equation}\label{dt-ralpha}
\begin{aligned}
\frac{d}{dt}  \int_{\mathbb{T}^d}
\Im ( r_\alpha \bar r) \, d\balpha = & \ 
 2\Re \int_{\mathbb{T}^d}    \bar{P}^\sharp\left[ (1+a)(1-\bar{Y})^2 \bar{w} \right]r_\alpha 
 + 2\Im ( r_\alpha \bar g) \, d \balpha,
\end{aligned}
\end{equation}

\begin{equation}
\begin{aligned}
\frac{d}{dt}  \int_{\mathbb{T}^d}
\vert r\vert^2 \, d\balpha =\  & 2\Re \int_{\mathbb{T}^d}\left( 
\mD_t r\,  - br_{\alpha} \right) \bar{r}   \, d\balpha \\
=\  &2\Re \int_{\mathbb{T}^d}
\left( P^{\sharp} [g] +i\left[ (1+a)(1-Y)^2w\right]  \right) \bar{r}  \, d\balpha  + \int_{\mathbb{T}^d} b_{\alpha} \vert r\vert^2 \, d\balpha .
\end{aligned}
\end{equation}

To estimate the time derivative of the energy 
we consider each of the terms above separately, 
except for the fourth term in \eqref{dt-w}
and the first term in \eqref{dt-ralpha}, which we 
cancel modulo a commutator,
\[
-2 \Re \int_{\T^d} (1+a) |1-Y|^2 \bar w \cdot [P^\sharp,(1-\bar Y)(1+\W)] r_\alpha \, d\balpha.
\]

We need several uniform bounds,
namely 
\begin{equation}
\label{b-est0}
\|b_\alpha \|_{L^\infty} \lesssim_A B,
\qquad \|D_t((1+a)|1-Y|^2) \|_{L^\infty}
\lesssim_A B,
\end{equation}
as well as two commutator bounds,
\begin{equation}
\label{b-est}
\| [b,P^{\sharp}]\partial_\alpha\|_{L^2 \to L^2}
\lesssim_A B, 
\end{equation}
\begin{equation}\label{YW-com}
   \| [(1-\bar Y)(1+\W),P^{\sharp}]\partial_\alpha\|_{H^{0,\frac12} \to L^2}\lesssim_A B.  
\end{equation}

The first bound in \eqref{b-est0} follows directly from Lemma \ref{lem:b}.  
The second bound requires the material derivative of $a$, which is a direct application of part (v) of Lemma~\ref{lem:a}, and the material derivative of $Y$, or equivalently $\W$, which can be estimated directly from the $\W$ equation. 

We now consider the bounds in \eqref{b-est}. The commutator bound involving $b$ can be obtained as a consequence of Lemma~\ref{l:comb} and Lemma~\ref{lem:b}. Finally, for \eqref{YW-com} we estimate
\[
\| (1-\bar Y)(1+\W) - 1\|_{H^s}
\lesssim_A B  
\]
and then use the second part of Lemma~\ref{l:comb}.

Next we consider the contribution of the source terms to the time derivative of the energy, namely   
\begin{equation}
\int_{\T^d}2(1+a)| 1-Y|^2\Re \left\{ P^\sharp [f] \bar{w}\right\} + 2\Im \left\{ P^\sharp [g] \bar{r}_{\alpha}\right\} +2\Re \left\{ P^{\sharp} [g]\bar{r} \right\}\, d\balpha.
\label{e}
\end{equation}
This can be estimated directly as needed in the theorem.
\end{proof}

We now return to prove Theorem~\ref{plin-short}, which concerns the fully linearized equation,
and show that it follows from Theorem~\ref{plin-reduced}.

\begin{proof}[Proof of Theorem~\ref{plin-short}]
By Theorem~\ref{plin-short}, it will suffice to prove  that the source terms $(f,g)$ in \eqref{e:source} satisfy the bounds
\begin{equation}\label{lin-source}
\|(P^{\sharp}  f,P^{\sharp}  g)\|_{\H^0} 
\lesssim_ A B \|(w, r)\|_{\H^0} .
\end{equation}
We begin by inspecting each term individually.

\bigskip

\emph{ The bound for $f$.}
For the first term in $ f$ we simply use Sobolev embeddings
\[
\| (1 - \bar Y) R_\alpha  w\|_{L^2}
\lesssim \| (1 - \bar Y) R_\alpha\|_{L^\infty}
\|  w\|_{L^2} \lesssim (1+A)B \| w\|_{L^2}.
\]
The second term in $ f$ is similar and the third term can be estimated using Lemma \ref{lem:M}.

For the last term in \eqref{lin(hwhr)} we need the estimate
\[
\| \delta M\|_{L^2} \lesssim_A B \|( w, r)\|_{\H^0}.
\]
Recall that
  \[
\delta M = 2\Re P[ r \bar Y_\alpha -2 R(1 - \bar Y) \bar{ w} \bar Y_\alpha + R (1 - \bar Y)^2 \bar{ w}_\alpha - \bar{ r}_\alpha Y - \bar R_\alpha (1 - Y)^2 w].
\]
The fifth term can be bounded using a standard bilinear estimate.  The 
others are similar to each other. We discuss the third, which is most interesting. Using a Littlewood-Paley decomposition in the $\alpha$ direction, we can write
\[
P[ R (1 - \bar Y)^2 \bar{ w}_\alpha]
= \sum_{\mu \leq \lambda}
P[ R_\lambda (1 - \bar Y)^2 \bar{ w}_{\mu,\alpha}],
\]
using the fact that $\bar Y$ is antiholomorphic so the $R$ frequency in the $\alpha$ direction must be at least as large as that of $w$.
Then we estimate in $L^2$ using Bernstein's inequality \eqref{b2} for $R_\lambda$, and pointwise bounds for $Y$ as follows,
\[
\begin{aligned}
\| P[ R (1 - \bar Y)^2 \bar{w}_\alpha]\|_{L^2} &\lesssim_A \sum_{\mu \leq \lambda}
\|R_\lambda\|_{L^\infty} \mu \|\bar{w}_{\mu}\|_{L^2} \\
&\lesssim_A \sum_{\mu \leq \lambda}
\lambda^{-1}\|R_\lambda\|_{H^{s,\frac12}} \mu \|{ w}_{\mu}\|_{L^2} \\
&\lesssim_A \|R\|_{H^{s,\frac12}}\|{ w}\|_{L^2}.
\end{aligned}
\]
It remains to consider the fourth term in $ f$, namely
$ \W_{\alpha} \delta b $. We recall that 
\[
\delta b =2\Re P[(1 - \bar Y)  r - (1 - \bar Y)^2 R \bar{ w}].
\]
We can estimate $\delta b$ in $H^{0,\frac12}$,
\begin{equation}\label{deltab}
\| \delta b\|_{H^{0,\frac12}} \lesssim_{A} 
A  \|( w, r)\|_{\H^0}.
\end{equation}
To see this we consider for instance the second term, where we use again a Littlewood-Paley decomposition in the $\alpha$ direction to write 
\[
P_\lambda P[ (1 - \bar Y)^2 R \bar{  w}]
= \sum_{\mu \geq \lambda}P[ (1 - \bar Y)^2 R_\mu \bar{ w}].
\]
Then we estimate it by 

\begin{align*}
\| P[ (1 - \bar Y)^2 R \bar{  w}]\|_{H^{0,\frac12}} \lesssim_A 
 &
\sum_{\mu \geq \lambda} \lambda^{\frac12}\| R_\mu\|_{L^\infty} \| w\|_{L^2}\\
&\lesssim_A 
\| w\|_{L^2} \sum_{\mu \geq \lambda} \lambda^{\frac12} \mu^{\frac{d}2-s}   \| R_\mu\|_{H^{s-\frac12, \frac12}}  \\
&\lesssim_A \| R\|_{H^{s-\frac12,\frac12}} \| w\|_{L^2},
\end{align*}
where we use Lemma \ref{lem:Bernstein} in the second inequality.
The first term in $\delta b$ is estimated in a similar manner, and \eqref{deltab} follows. 
It remains to show that 
\begin{equation}\label{tricky-bi}
\| \W_{\alpha} \delta b\|_{L^2} 
\lesssim \| \delta b\|_{H^{0,\frac12}}
\|\W\|_{H^s}.
\end{equation}

Here $\W_\alpha \in H^{s-1}$ (controlled by $B$) with $s > \frac{d+1}2$. 
This follows from Lemma~\ref{l:comb1}.

\bigskip

\emph{ The bound for $g$.}
 For the first term we need to show that 
\begin{equation}\label{tricky-bi1}
\|R_\alpha \delta b\|_{H^{0,\frac12}} 
\lesssim \| R_\alpha\|_{H^{s-\frac12}}\| \delta b\|_{H^{0,\frac12}},
\end{equation}
which is similar to \eqref{tricky-bi}, using Lemma~\ref{l:comb1}.

For the second term in $ g$ we  
dispense with the $1-Y$ prefactor using \eqref{tricky-bi1}. Then it remains to show
that 
\[
\|\delta a\|_{H^{0,\frac12}} \lesssim B \|r\|_{H^{0,\frac12}}.
\]
We recall that 
\begin{equation*}
\delta a =  2 \Im P[\bar R_\alpha r + R\bar r_\alpha].
\end{equation*}
The first term satisfies the desired estimate due to  Lemma \ref{l:comb1}. We rewrite the second term as
\[
P[R \bar r_{\alpha}]= [R,P]\partial_\alpha \bar r,
\]
and estimate it using the commutator estimate (\ref{com3}) in Lemma \ref{l:comb}.

\end{proof}

\subsection{ Difference bounds and uniqueness}

Here we consider estimates for differences of solutions in the weaker topology $\H^0$. While 
morally equivalent to the estimates for the linearized equation, directly formalizing this relationship is not entirely straightforward. 
So instead we largely repeat the analysis for the linearized equation.

\begin{theorem}\label{thm:diff}
    Let $(\W^1,R^1)$ and $(\W^2,R^2)$ be two solutions for the differentiated water wave system \eqref{eq:W:R} with joint control parameters $A : = A_1+A_2$ and $B := B_1+B_2$. Then we have 
\begin{equation}
\| (\W^1-\W^2,R^1-R^2)(t)\|_{\H^0}
\lesssim \| (\W^1-\W^2,R^1-R^2)(0)\|_{\H^0}
e^{C(A) \int_0^t B(s) ds}.
\end{equation}
\end{theorem}
\begin{proof}
Subtracting the two sets of equations 
we obtain a system of equations for the difference, which is denoted by  $(w,r)$.
Introducing also the notations $\delta b$, $\delta a$ and $\delta M$ for the corresponding differences,  this system has the form
\begin{equation}\label{diff(hwhr)}
\left\{
\begin{aligned}
 &D^1_t w + (1 - \bar Y^1)(1 + \W^1) r_\alpha 
 = f, \\
&D_t^1 r  -i(1 + a^1)(1 - Y^1)^2 w = g,
\end{aligned}
\right.
\end{equation}
where the source terms $f$ and $g$ are given by 
\begin{equation}\label{diff:source}
\left\{
\begin{aligned}
&f =  - (1-\bar Y^2)R^2_\alpha w + (1-\bar Y^1)(1-\bar Y^2)(1+\W^1)R^2_\alpha \bar w + M^1 w -\W_\alpha^2 \delta b + (1+\W^2)\delta M, \\
&g = - R^2_\alpha \delta b - i(1-Y^2)\delta a
- i(1+a)(1-Y^1)^2(1-Y^2) w^2 . 
\end{aligned}
\right.
\end{equation}
By Theorem~\ref{plin-reduced},  the energy estimates for the difference follow by Gronwall's inequality provided we can prove a favourable estimate for the source terms,
\begin{equation}\label{diff-source}
\|(P^{\sharp}  f,P^{\sharp}  g)\|_{\H^0} 
\lesssim_ A B \|(w, r)\|_{\H^0} .
\end{equation}

We carefully wrote the expressions for
the source terms in a manner closely resembling
the linearized equation source terms. The estimates are also similar. 
In particular the difference estimates for  $\delta b$, $\delta a$ and $\delta M$ are identical to the ones proved for the linearized equation, and  in effect a direct consequence of those. Consequently, we focus on the single term that did not appear in the linearized equation, namely the last term in $g$, which is quadratic in $w$, and for which we need to show that 
\[
\| (1+a)(1-Y^1)^2(1-Y^2) w^2\|_{H^{0,\frac12}}
\lesssim_A B \|w\|_{L^2}.
\]
Here the $L^2$ norm is weaker than the 
output $H^{0,\frac12}$ norm, so a corresponding linear in $w$ bound would not work. It is of the essence then that this term is quadratic in $w$.

By algebra properties and Lemma~\ref{lem:a} (ii) the prefactor is harmless, 
\[
\| (1+a)(1-Y^1)^2(1-Y^2) - 1\|_{H^{s-\frac12}} \lesssim_A 1,
\]
and can be discarded by Lemma~\ref{l:comb1}.
Hence, since $H^{\frac12} \subset H^{0,\frac12}$, it suffices to show that 
\[
\| w^2\|_{H^{\frac12}}  \lesssim \|w\|_{L^2} \|w\|_{H^s}.
\]
Since $s > \frac{d+1}2$, this is now a standard 
bilinear estimate.
\end{proof}

\section{Higher order energy estimates} 
\label{s:ee}
The main goal of this section is to establish two energy bounds for
$(\W,R)$ and their higher derivatives. This is a large data result:

\begin{theorem}
\label{t:ee}
For each nonnegative integer $k\geq 1$ there exists an energy $E^k(\W,R)$ 
so that 

\begin{enumerate}
    \item Coercivity 
\begin{equation}\label{Ek-coercive}
 E^k(\W,R) \approx \| (\W,R)\|_{\H^k}^2.   
\end{equation}
\item Growth bound
\begin{equation}\label{Ek-dt}
   \left|\frac{d}{dt}  E^k(\W,R)\right| \lesssim_A B \| (\W,R)\|_{\H^k}^2. 
\end{equation}
\end{enumerate}
\end{theorem}
We remark that the same result can also be proved for $k=0$ via a minor variation of the arguments below.
\begin{proof}
We differentiate the system $k$ times. To do this we consider any multi-index $\kappa$ of length $k$. 
We denote the
 differentiated variables by $(\W^{[\kappa]}, R^{[\kappa]})$. 
Then we define 
\begin{equation}
E^k(\W,R):= \sum_{|\kappa| = k}
E_{lin}(\W^{[\kappa]}, R^{[\kappa]}).
\end{equation}
Then the coercivity property \eqref{Ek-coercive} is immediate, and it remains to prove \eqref{Ek-dt}.

We consider the equations
\eqref{ww2d-diff-real}, which we restate here, and rewrite the nonlinear terms in the equations as multilinear expressions 
in $\W,R$ and $Y$, 
\begin{equation} \label{eq:WW2d:diff:Y}
\left\{
\begin{aligned}
 & \W_{ t} +  b\W_{ \alpha} + (1+\W)(1-\bar Y) R_\alpha   =  (1+\W)M,
\\
& R_t + b R_\alpha  - i(1-Y)(1+a) = i,
\end{aligned}
\right.
\end{equation}
where we recall that 
\[
b = 2\Re P[R(1-\bar Y)], \quad 
a = \bar P \left[\bar{R} R_\alpha\right]- P\left[R\bar{R}_\alpha\right],
\quad
M = \bar P [\bar R Y_\alpha- R_\alpha \bar Y]  + P[R \bar Y_\alpha - \bar R_\alpha Y].
\]

To track the evolution of $E^k(\W,R)$, we differentiate both equations in \eqref{eq:WW2d:diff:Y} $k$ times and obtain
a system for $(\W^{[\kappa]}, R^{[\kappa]})$, which we write as a reduced linearized system 
with source terms:
\begin{equation} \label{eq:WW:diff:k}
\left\{
\begin{aligned}
 & D_t \W^{[\kappa]} + (1 - \bar Y)(1 + W_\alpha) R^{[\kappa]}_\alpha 
 =  F^{[\kappa]},
\\
&D_t R^{[\kappa]}   -i(1 + a)(1 - Y)^2\W^{[\kappa]} =  G^{[\kappa]}.
\end{aligned}
\right.
\end{equation}
We remark that the reduced linearized system is obtained when 
all $k$ derivatives  fall on the same differentiated $\W$ or $R$ factor, as well as on $Y$ in the second equation,
while the source terms correspond to distributed derivatives. Here we also have $Y$ factor,  which requires special care; we differentiate it as follows: 
\[
\partial^k Y = \partial^{k-1}( \partial \W (1-Y)^2)
= \partial^{k} \W (1-Y)^2 + 
\partial^{k-2} (\partial \W \partial (1-Y)^2).
\]
Only the first term contributes to the reduced 
linearized system, and the rest contribute to the source term.

Overall, for $F^{[\kappa]}$ and $G^{[\kappa]}$ we obtain the expressions
\[
\begin{aligned}
F^{[\kappa]} = & \ -\partial^{k-1} (\partial b\W_{ \alpha}) - \partial^{k-1} 
(\partial[(1+\W)(1-\bar Y)] R_\alpha) + \partial^k[(1+\W) M],
\\
G^{[\kappa]} = & \ -\partial^{k-1} (\partial b R_{ \alpha})
+ i(1+a)  
\partial^{k-2} (\partial \W \partial (1-Y)^2) + i \partial^{k-1}[ 
\partial a (1-Y)].
\end{aligned}
\]

For the system \eqref{eq:WW:diff:k}, we can apply the 
energy estimates for the reduced linearized 
equations in Theorem~\ref{plin-reduced}, therefore it suffices to bound the source terms in $\H^0$,
\begin{equation}\label{FG-kappa}
\|(F^{[\kappa]},G^{[\kappa]})\|_{\H^0} \lesssim_A B \|(\W,R)\|_{\H^k}.
\end{equation}
We split the proof of the 
bounds \eqref{FG-kappa} into two steps:

\begin{itemize}
    \item We first prove bounds
    for the auxiliary variables 
    $Y,b,a,M$, where the commutator structure involving the projectors is 
    essential.

    \item It then remains to prove bilinear estimates for each term in $F^{[\kappa]}$ and $G^{[\kappa]}$, where the holomorphic structure can be safely neglected.
\end{itemize}

The bounds for $Y$, $b$, $a$, $M$ have been separately proved 
in Lemmas~\ref{l:Y}, \ref{lem:b}, \ref{lem:a} and \ref{lem:M}, so it remains 
to estimate the expressions above in a bilinear fashion.

Since we have bounds for each entry in $H^s$ type Sobolev spaces and we want to estimate the output also in $H^s$ type Sobolev spaces we can measure all norms in terms of the size of the respective Fourier coefficients.
By writing multiplications as convolutions in the Fourier space, it is easily seen that the worst-case scenario is when all Fourier coefficients are positive. We can also estimate 
distributed  derivatives by considering only the endpoint cases. 

\bigskip

a) The $L^2$ bound for  
$\partial^{k-1} (\partial b\W_{ \alpha})$. We consider the two extreme cases, using the product  bound \eqref{prod2} in Lemma~\ref{l:comb1} and the $b$ bounds in Lemma~\ref{lem:b}. First  we have 
\[
\| \partial^k b\W_{ \alpha}\|_{L^2} \lesssim \|\partial^k b\|_{H^{0,\frac12}} \|\W_{\alpha}\|_{H^{s-1}} \lesssim_A B E^k(\W,R)^\frac12.
\]

Secondly,
\[
\| \partial b\partial^{k-1} \W_{ \alpha}\|_{L^2} \lesssim 
\| \partial b\|_{L^\infty} 
\| W\|_{H^k} \lesssim \| b\|_{H^{s,\frac12}} \| W\|_{H^k}
\lesssim_A B E^k(\W,R)^\frac12.
\]
\bigskip

b) The $L^2$ bound for $\partial^{k-1} 
(\partial[(1+\W)(1-\bar Y)] R_\alpha$). Here 
by Lemma~\ref{l:Y}, the expression $\W_1:= \W - \bar Y -\W\bar Y$ satisfies the same bounds as $\W$,
\[
\|\W_1\|_{H^s} \lesssim_A \|\W \|_{H^s}, 
\qquad
\|\W_1\|_{H^k} \lesssim_A \|\W \|_{H^k}. 
\]
To bound $\partial^{k-1} (\partial \W_1 R_\alpha)$ in $L^2$ we first consider the extreme cases.  
For the first extreme case we have 
\[
\| \partial^k \W_1 R_\alpha\|_{L^2} \lesssim \| \partial^k \W_1\|_{L^2} \| R_\alpha\|_{L^\infty} \lesssim \|\W_1\|_{H^k} \| R\|_{H^{s,\frac12}}
\lesssim_A B E^k(\W,R)^\frac12. 
\]
For the second extreme case
we have
\[
\| \partial \W_1 \partial^{k-1} R_\alpha\|_{L^2} \lesssim \| \partial \W_1 \|_{H^{s-1}} \| \partial^{k-1} R_\alpha\|_{H^{0,\frac12}} \lesssim \|\W_1\|_{H^s} \| R\|_{H^{k,\frac12}}
\lesssim_A B E^k(\W,R)^\frac12.
\]
The intermediate cases follow from a standard argument interpolating between the extreme cases. 
\bigskip

c) The $L^2$ bound for $\partial^k[(1+\W) M]$. Here we use Lemma~\ref{lem:M}. For the first extreme case we have
\[
\| \partial^k\W M\|_{L^2}
\lesssim \|\W\|_{H^k} \|M\|_{H^{s-\frac12}} \lesssim_A B E^k(\W,R)^\frac12,
\]
and the second is symmetric.

\bigskip

d) The $H^{0,\frac12}$ bound for $\partial^{k-1} (\partial b R_{ \alpha})$. This is similar to a). First,  by the bound \eqref{prod3} in Lemma \ref{l:comb1} and the bounds for $b$ in Lemma \ref{lem:b}, we have 
\[
\| \partial^k bR_{ \alpha}\|_{H^{0,\frac12}} \lesssim \|\partial^k b\|_{H^{0,\frac12}} \|R_{\alpha}\|_{H^{s-\frac12}} \lesssim_A B E^k(\W,R)^\frac12.
\]
Secondly,
\[
\| \partial b\partial^{k-1} R_{ \alpha}\|_{H^{0,\frac12}} \lesssim 
\| \partial b\|_{H^{s-1,\frac12}} 
\| \partial^{k-1} R_{ \alpha}\|_{H^{0, \frac12}} \lesssim \| b\|_{H^{s,\frac12}} \| R\|_{H^{k,\frac12}}
\lesssim_A B E^k(\W,R)^\frac12. 
\]

\bigskip

e) The $H^{0,\frac12}$ bound for $(1+a)  \partial^{k-2} (\partial \W \partial (1-Y)^2)$. 
Here we have 
\[
\begin{aligned}
\| (1+a) \partial (1-Y)^2 \partial^{k-1} \W\|_{H^{0,\frac12}}\lesssim & \ (1+\|a\|_{H^{s-\frac12}}) \| \partial (1-Y)^2 \|_{H^{s-1}} \| \partial^{k-1} \W\|_{H^1}
\\
\lesssim_A & \ \|Y\|_{H^s} \|\W\|_{H^{k}}\lesssim_A B E^k(\W,R)^\frac12,
\end{aligned}
\]
where we use the bounds for $a$ in Lemma~\ref{lem:a}. The second extreme case is nearly identical.

\bigskip

f) The $H^{0,\frac12}$ bound for $
\partial^{k-1}[ \partial a (1-Y)]$.
Here on the one hand we use Lemma~\ref{l:comb1} and Lemma~\ref{lem:a} to write
\[
\| \partial^k a (1-Y)\|_{H^{0,\frac12}}
\lesssim \|\partial^k a\|_{H^{0,\frac12}} (1+\|Y\|_{H^{s-\frac12}}) \lesssim_A B E^k(\W,R)^\frac12.
\]
On the other hand we have
\[
\| \partial a \partial^{k-1} Y\|_{H^{0,\frac12}}
\lesssim \|\partial a\|_{H^{s-1}} \|\partial^{k-1}  Y\|_{H^{1}} \lesssim_A B E^k(\W,R)^\frac12.
\]

\end{proof}
\section{Multiplicative and commutator bounds}
\label{s:multi}

Here we collect all multilinear and commutator estimates needed for the proof. 
We begin with a Bernstein type inequality for the projectors $P_\lambda ^\alpha$:

\begin{lemma}
\label{lem:Bernstein}
Assume that $s > \frac{d-1}2$. Then we have
\begin{equation}
\label{b1}
\|P_\lambda ^\alpha u\|_{L^\infty} \lesssim \lambda^{\frac{d}2-s} \| u\|_{H^s},     
\end{equation}
as well as 
\begin{equation}
\label{b2}
\|P_\lambda ^\alpha u\|_{L^\infty} \lesssim \lambda^{\frac{d-1}2-s} \| u\|_{H^{s,\frac{1}{2}}}.   
\end{equation}
\end{lemma}
\begin{proof}
For the bound \eqref{b1} we use the Fourier series  expansion 
\[
P_\lambda^\alpha u= \sum_{\bfj \in D_\lambda} \hat u_{\bfj} e^{i \bfj\cdot \balpha},
\]
where 
\[
D_\lambda = \{ \bfj \in \Z^d;\ |\langle \bfj,\bfk\rangle| \approx \lambda\}
\]
and estimate 
\[
\begin{aligned}
\| P_\lambda^\alpha u \|_{L^\infty} &= \sum_{\bfj \in D_\lambda} |u_{\bfj}|
\lesssim \sum_{\bfj \in D_\lambda} |u_{\bfj}|\langle\bfj\rangle ^{s}
\cdot \frac{1}{\langle\bfj\rangle^{s}} \\ &\lesssim \left( \sum_{\bfj \in D_\lambda} |u_{\bfj}|^2 \langle \bfj\rangle ^{2s}\right)^\frac12 \left( \sum_{\bfj \in D_\lambda} \frac{1}{\langle \bfj\rangle ^{2s}}\right)^\frac12.
\end{aligned}
\]
We can replace the last sum with an integral, where we need to bound
\[
\int_{ |\bfj \cdot \bfk| \approx\lambda} \frac{1}{\langle \bfj\rangle ^{2s}} d \bfj.
\]
Here we can freely rotate coordinates so that $\bfk = e_1$ 
where the integral becomes
\[
\int_{|\xi_1| \approx\lambda} \frac{1}{\langle \xi\ra^{2s}} d\xi
\lesssim \int_{|\xi| \approx \lambda} \frac{1}{\langle \xi\rangle^{2s}} d\xi
+ \lambda \int_{|\xi'| > \lambda} \frac{1}{\langle \xi'\rangle^{2s}} d\xi' \approx \lambda^{\frac d2-s}.
\]
These integrals are both easy to compute in polar coordinates, and the requirement $s >\frac{d-1}2$ is needed in order to guarantee the convergence of the second.

The bound (\ref{b2}) follows in an almost identical manner.

\end{proof}

We continue with some commutator bounds:

\begin{lemma}\label{l:comb}
 For $s > \frac{d+1}2$ we have the following commutator bounds:
\begin{equation}\label{com1}
\| [f,P] \partial_\alpha\|_{L^2 \to L^2} \lesssim \| f\|_{H^{s,\frac12}},
\end{equation}
\begin{equation}\label{com2}
\| [f,P] \partial_\alpha\|_{H^{0,\frac12} \to L^2} \lesssim \| f\|_{H^{s}},
\end{equation}
\begin{equation}\label{com3}
\| [f,P] \partial_\alpha\|_{H^{0,\frac12}\to H^{0,\frac12}} \lesssim \| f\|_{H^{s,\frac{1}{2}}}.
\end{equation}
\end{lemma}
We note that the same bounds also hold for $s = \frac{d+1}2$, with slight enhancements in the proof below.
\begin{proof}
We use a Littlewood-Paley decomposition in the $\alpha$ direction, decomposing the commutator as 
\[
[f,P] \partial_\alpha u = \sum_{\mu < \lambda}
(f_\lambda P \partial_\alpha u_\mu - Pf_\lambda \partial_\alpha u_\mu) - \sum_\lambda P( f_\lambda \partial_\alpha u_\lambda).
\]

Then for \eqref{com1} we can use Lemma \ref{lem:Bernstein} to estimate
\[
\begin{aligned}
\|[f,P] \partial_\alpha u\|_{L^2} \lesssim & \sum_{\mu \leq \lambda}  
\| f_\lambda \|_{L^\infty} \| \partial_\alpha u_\mu \|_{L^2} 
\lesssim  \sum_{\mu \leq \lambda}  
\| f_\lambda \|_{L^\infty}  \mu \| u_\mu \|_{L^2} 
\\
\lesssim & \sum_{\mu \leq \lambda} \lambda^{\frac{d-1}2-s}  \| f \|_{H^{s,\frac12}}  \mu \| u_\mu \|_{L^2} 
\\
\lesssim & \ \| f \|_{H^{s,\frac12}}   \| u \|_{L^2}.
\end{aligned}
\]
For \eqref{com2} we fork the estimate at the second line, writing instead
\[
\begin{aligned}
\|[f,P] \partial_\alpha u\|_{L^2} 
\lesssim & \sum_{\mu \leq \lambda} \lambda^{\frac{d}2-s}  \| f \|_{H^{s}}  \mu^\frac12 \| u_\mu \|_{H^{0,\frac12}} 
\\
\lesssim & \| f \|_{H^{s}}   \| u \|_{H^{0,\frac12}}.
\end{aligned}
\]
The argument for \eqref{com3} is similar and is left for the reader.
This concludes the proof of the Lemma.
\end{proof}

\begin{lemma}\label{l:comb1}
 For $s>\frac{d+1}{2}$ we have the multiplication bounds 

 \begin{equation}\label{prod1}
\| f   \|_{H^{0,\frac12} \to H^{0,\frac12}} \lesssim \| f\|_{H^{s-\frac12},}
\end{equation}
\begin{equation}\label{prod2}
\| f\|_{H^{0,\frac12} \to L^2} 
\lesssim
\| f\|_{H^{s-1}},
\end{equation}
\begin{equation}
\label{prod3}
\|f\|_{H^{0,\frac12} \to H^{0, \frac12}} \lesssim \|f\|_{H^{s-1, \frac12}}.
\end{equation}
\end{lemma}

\begin{proof}
We note that the bound \eqref{prod3} implies the bound in \eqref{prod1}. Therefore, it suffices to prove \eqref{prod2} and \eqref{prod3}.
We use again a Littlewood-Paley decomposition in the $\alpha$ direction, decomposing the product in a paradifferential fashion as 
\[
fu = \sum_{\lambda} f_{< \lambda} u_{\lambda} + f_\lambda u_\lambda 
+ f_\lambda u_{<\lambda} := S_1+S_2+S_3.
\]
In $S_1$ the output of the summands is localized at $\alpha$-frequency $\lambda$, therefore for \eqref{prod3} we can use orthogonality and Bernstein's inequality to write
\[
\begin{aligned}
\|S_1\|_{H^{0,\frac12}}^2 \lesssim & \sum_\lambda \lambda \|f_{< \lambda} u_{\lambda}\|_{L^2}^2 \lesssim  \sum_\lambda \lambda \|f_{< \lambda}\|_{L^\infty}^2 \| u_{\lambda}\|_{L^2}^2 
\\ \lesssim & \ 
\|f\|_{H^{s-1, \frac12}}^2\sum_\lambda \lambda \| u_{\lambda}\|_{L^2}^2 
\lesssim 
\|f\|_{H^{s-1, \frac12}}^2 \|u\|_{H^{0,\frac12}}^2.
\end{aligned}
\]
Similarly for \eqref{prod2} we have
\[
\begin{aligned}
\|S_1\|_{L^2}^2 \lesssim & \sum_\lambda  \|f_{< \lambda} u_{\lambda}\|^2_{L^2} \lesssim  \sum_\lambda \lambda^{-1} \|f_{< \lambda}\|_{L^\infty}^2 \lambda \| u_{\lambda}\|_{L^2}^2 
\\ \lesssim & \ 
\|f\|_{H^{s-1}}^2\sum_\lambda \lambda \| u_{\lambda}\|_{L^2}^2 
\lesssim 
\|f\|_{H^{s-1}}^2 \|u\|_{H^{0,\frac12}}^2.
\end{aligned}
\]
The bound for $S_2$ is simpler, as we no longer need to use orthogonality. To estimate $S_2$ for \eqref{prod3} we write
\[
\begin{aligned}
\|S_2\|_{H^{0,\frac12}} \lesssim & \sum_\lambda \lambda \|f_{ \lambda} u_{\lambda}\|_{L^2} \lesssim  \sum_\lambda \lambda \|f_{ \lambda}\|_{L^\infty} \| u_{\lambda}\|_{L^2}
\\ \lesssim & \ 
\sum_\lambda \lambda^{\frac{d+1}2 - s} \|f_\lambda\|_{H^{s-1,\frac12}} \| u_{\lambda}\|_{H^{0,\frac12}}
\lesssim 
\|f\|_{H^{s-1, \frac12}} \|u\|_{H^{0,\frac12}}
,
\end{aligned}
\]
while for \eqref{prod2} we have
\[
\begin{aligned}
\|S_2\|_{L^2} \lesssim & \sum_\lambda \|f_{ \lambda} u_{\lambda}\|_{L^2} \lesssim  \sum_\lambda  \|f_{ \lambda}\|_{L^\infty} \| u_{\lambda}\|_{L^2}
\\ \lesssim & \ 
\sum_\lambda \lambda^{\frac{d+1}2 - s} \|f_\lambda\|_{H^{s-1}} \| u_{\lambda}\|_{H^{0,\frac12}}
\lesssim 
\|f\|_{H^{s-1}} \|u\|_{H^{0,\frac12}}
.
\end{aligned}
\]
The proof of the bounds for $S_3$ is similar and is omitted.

\end{proof}

The next lemma uses the paraproduct notion,
which arises when one considers
paraproduct  decomposition  of the product of two functions,
\[
 f g = \sum_{k > l+4} f_{l} g_k +  \sum_{k > l+4} f_{k} g_{l} + \sum_{|k-l| \leq 4}
f_k g_l := T_f g + T_g f + \Pi(f,g).
\]

With this notation, we have
\begin{lemma}\label{l:para-err}
For $s > \frac{d+1}2$ we have 
\begin{equation}
\| (T_f - f) u\|_{H^{0,\frac12}} \lesssim 
\| f\|_{H^s} \|u\|_{L^2}.
\end{equation}
\end{lemma}

\begin{proof}
With a Littlewood-Paley decomposition in the $\alpha$ direction, the 
expression to estimate can be expanded as 
\[
(T_f - f) u = \sum_{\mu \leq \lambda} f_\lambda u_\mu.
\]
Then we estimate
\[
\begin{aligned}
\|  (T_f - f) u\|_{ H^{0,\frac12}} \lesssim &  \sum_{\mu \leq \lambda} \lambda^\frac12 \|f_\lambda u_\mu\|_{L^2} 
\\
\lesssim &  \sum_{\mu \leq \lambda} \lambda^\frac12 \|f_\lambda\|_{L^\infty}  \|u_\mu\|_{L^2} 
\\
\lesssim &  \sum_{\mu \leq \lambda} \lambda^{\frac{d+1}2-s} \|f_\lambda\|_{H^s}  \|u_\mu\|_{L^2} 
\\
\lesssim &  \|f\|_{H^s}  \|u\|_{L^2} \sum_{\lambda} \lambda^{\frac{d+1}2-s} 
\\
\lesssim &  \|f\|_{H^s}  \|u\|_{L^2}
\end{aligned}
\]
as needed.
\end{proof}

\section{Water Waves related estimates}
\label{s:water}

\begin{lemma}\label{l:Y}
Assume $s > \frac{d+1}2$. Then the function $Y$ satisfies the following Moser-type estimates:
\begin{itemize}
\item[a)] Let $k \geq 0$. Then in $H^k$ we get
\begin{equation}
\Vert Y\Vert_{H^\kappa} \lesssim_A\Vert\W\Vert_{H^k}.
\end{equation}
\item[b)] In $H^s$ we get
\begin{equation}
\Vert Y\Vert_{H^s} \lesssim_A\Vert\W\Vert_{H^s}.
\end{equation}
\item[c)] In $H^{s-\frac{1}{2}}$ we get
\begin{equation}
\Vert Y\Vert_{H^{s-\frac{1}{2}}} \lesssim_A\Vert\W\Vert_{H^{s-\frac{1}{2}}}.
\end{equation}
\end{itemize}
\begin{proof}
We note that proving $a)$ suffices. We begin with the definition of $Y$,
\[
Y:=\frac{\W}{1+\W},
\]
which is a smooth function of $\W$. So as long as $\W$ stays away from $-1$, we have the Moser estimate
\[
\|Y\|_{H^k} \leq C(\|\W\|_{L^\infty}, \| Y\|_{L^\infty}) \|\W\|_{H^k},
\]
which suffices by Sobolev embeddings. 
\end{proof}

\end{lemma}
\begin{lemma} \label{lem:b}
Assume $s > \frac{d+1}2$. Then the advection velocity $b$  satisfies the following Sobolev bounds
\begin{itemize}
\item[(i)] In terms of the control norm $A$ 
\begin{equation}\label{b-A}
\|b\|_{H^{s-\frac{1}{2},\frac{1}{2}}} \lesssim_A A , 
\end{equation}

\item[(ii)] In terms of the control norm $B$
\begin{equation}\label{b-B}
\|b\|_{H^{s,\frac{1}{2}}} \lesssim_A B, 
\end{equation}
\item[(iii)] In terms of the energy bound
\begin{equation*}
\|b\|_{H^{k,\frac{1}{2}}} \lesssim _A A(E^k)^{\frac{1}{2}}, \qquad k \geq 0.
\end{equation*}
\end{itemize}
\end{lemma}
\begin{proof}
It suffices to prove (iii). We recall the definition of $b$
\begin{equation}
b = \Re \left(R - P[R\bar Y]\right).
\end{equation}
There is nothing to do for the first term. For the second term we harmlessly 
assume that the two factors have nonnegative Fourier transforms, which allows us to write
\begin{align}
\| \langle D_\alpha\rangle^{1/2} P[R\bar Y]\|_{H^k} \lesssim & 
\| \la D\ra^{k}  \langle D_\alpha\rangle^{1/2} P[R\bar Y]\|_{L^2}
\\
\lesssim & \|\la D \ra^{ k}  \langle D_\alpha\rangle^{1/2} R \bar Y\|_{L^2} +  \|\langle D_\alpha\rangle^{1/2} R \la D\ra^{ k} \bar Y\|_{L^2}
\\
\lesssim & \|\la D\ra^{ k} \langle D_\alpha\rangle^{1/2} R\|_{L^2} \|\bar Y\|_{L^\infty} +  \| \langle D_\alpha\rangle^{1/2}  R\|_{L^\infty} \|\la D\ra^{k} \bar Y\|_{L^2}
\\
\lesssim & \| R\|_{H^{k,\frac12}} \| Y\|_{H^{s-\frac12}} +  \| R\|_{H^{s-\frac12,\frac12}} \| Y\|_{H^k}
\end{align}
as needed.

\end{proof}

\begin{lemma} \label{lem:a} 
Assume $s > \frac{d+1}2$.
Then for the quadratic expression $a$, defined in \eqref{defa},  the following bounds hold:

\begin{itemize}
\item[(i)] Positivity:
\[ 
a\geq 0.
\]
\item[(ii)] In terms of the control norm $A$ we have
\begin{equation}\label{a:A}
\|a\|_{H^{s-\frac12}} \lesssim A^2. 
\end{equation}
\item[(iii)] In terms of the control norm $B$ we have the bound
\begin{equation}\label{a:B}
\|a\|_{H^s} \lesssim AB. 
\end{equation}
\item[(iv)]  In terms of the energy functional the following holds
\[
\|a\|_{H^{k,\frac{1}{2}}} \lesssim  B(E^k)^{\frac{1}{2}}.
\]
\item[(v)] Material derivative: 
\[
\| D_t a\|_{L^\infty} \lesssim_A B.
\]
\end{itemize}
\end{lemma}

\begin{proof}
Recall the definition of the real frequency-shift $a$  
\begin{equation}
a := i(\bar P[\bar R R_\alpha] - P[R\bar R_\alpha]) = 2\Im P[R\bar R_\alpha].
\end{equation}
(i) Here we adapt the proof in \cite{HIT} to the quasi-periodic setting. 
We denote $\mu_\bfj := \la \bfj, \bfk \ra$ and express $a$ by
\[
\begin{aligned}
a = &\sum_{\mu_{\bds j_1}<0, \mu_{\bds j_2} <0} \min(|\mu_{\bds j_1}|, |\mu_{\bds j_2}|) 
 \hat{R}_{\bds j_1} \bar{\hat{R}}_{\bds j_2} e^{i(\bds j_1 - \bds j_2)\cdot \balpha} \\
= & \sum_{\mu_{\bds j_1}<0, \mu_{\bds j_2} <0}\int_{0}^\infty 1_{\{M \leq |\mu_{\bds j_1}|\}} 1_{\{M \leq |\mu_{\bds j_2}|\}}\; dM  \hat{R}_{\bds j_1} \bar{ \hat{R}}_{\bds j_2} e^{i(\bds j_1 - \bds j_2)\cdot \balpha} \\
= & \int_0^\infty \left|\sum_{\mu_{\bds j} < 0} 1_{\{M \leq |\mu_{\bds j}|\}}  \hat{R}_{\bds j}\right|^2 \; dM \geq 0.
\end{aligned} 
\]
Therefore, we know that $a$ is non-negative.

\bigskip
(ii) Assuming that the two factors have nonnegative Fourier transforms
and using the frequency supports of the two factors and the projection, we write
\[
\begin{aligned}
\|P[R\bar R_\alpha]\|_{H^{s-\frac12}} 
\lesssim &\ \| |D_\alpha|^\frac12 R 
 |D_\alpha|^\frac12 \bar R \|_{H^{s-\frac12}} 
 \\
 \lesssim & \ \| |D_\alpha|^\frac12 R\|_{ H^{s-\frac12}}
 \||D_\alpha|^\frac12 \bar R \|_{H^{s-\frac12}} 
 \\
 \lesssim & \ \|R\|_{H^{s-\frac12,\frac12}}^2 \lesssim A^2.
\end{aligned}
\]

(iii) We argue as in the previous case but unbalance the frequency powers
\[
\begin{aligned}
\|P[R\bar R_\alpha]\|_{H^{s}} 
\lesssim &\ \| |D_\alpha|^\frac12 R 
 |D_\alpha|^\frac12 \bar R \|_{H^{s}} 
 \\
 \lesssim & \ \| |D_\alpha|^\frac12 R\|_{ H^{s-\frac12}}
 \| |D_\alpha|^\frac12 \bar R \|_{H^{s}}
 + \| |D_\alpha|^\frac12 R\|_{ H^{s}}
 \| |D_\alpha|^\frac12 \bar R \|_{H^{s-\frac12}}
 \\
 \lesssim & \ \|R\|_{H^{s-\frac12,\frac12}}\|R\|_{H^{s,\frac12}} \lesssim AB.
\end{aligned}
\]

(iv) In the same spirit, 
\[
\begin{aligned}
\|P[R\bar R_\alpha]\|_{H^{k,\frac12}} 
\lesssim &\ \| \langle D\rangle^k \la D_\alpha\ra^\frac12 
P[R\bar R_\alpha]\|_{L^2}
\\
\lesssim & \
\| \la D_\alpha\ra^\frac12 
P[ \langle D\rangle^k  R\bar R_\alpha]\|_{L^2}
+ \|\la D_\alpha\ra^\frac12 
P[R  \langle D\rangle^k   \bar{R}_\alpha ]\|_{L^2}
\\
\lesssim & \
\|  
\langle D\rangle^k \la D_\alpha\ra^\frac12 R\bar R_\alpha\|_{L^2}
+ \| 
\la D_\alpha \ra R  \langle D \rangle^k \la D_\alpha\ra^\frac12 \bar{ R}\|_{L^2}
\\
\lesssim & \ \|\langle D\rangle^k \la D_\alpha\ra^\frac12 R\|_{L^2} \| \la D_\alpha \ra R \|_{L^\infty}
\\
\lesssim & \| R\|_{H^{k,\frac12}}
\| R\|_{H^{s,\frac12}} \lesssim  B(E^k)^{\frac{1}{2}}.
\end{aligned}
\]

(v) We begin with the following computation
\[
\begin{aligned}
D_t a = & \  ( \partial_t + T_b \partial_\alpha) a + (b-T_b) \partial_\alpha a
\\
= & \ 2\Im P[ ( \partial_t + T_b \partial_\alpha) (R \bar R_\alpha)] + (b-T_b) \partial_\alpha a  
\\
= & \ 2\Im P[ ( \partial_t + T_b \partial_\alpha) R \bar R_\alpha] +
2\Im P[ R \partial_\alpha (\partial_t + T_b \partial_\alpha) \bar R] - 2\Im P[ R  T_{b_\alpha} \bar R_\alpha] \\
& \ 
+ 2\Im P[ T_b \partial_\alpha ( R \bar R_\alpha)
- T_b R_\alpha \bar R_\alpha
-  R T_b \bar R_{\alpha\alpha}]
+ (b-T_b) \partial_\alpha a.
\end{aligned}
\]
It suffices to bound each of the terms in $H^{s-\frac12}$. We begin with the first two terms, which involve  paramaterial derivatives, and are bounded by
\[
\|P[ ( \partial_t + T_b \partial_\alpha) R \bar R_\alpha]\|_{H^{s-\frac12}}+ \| P[ R \partial_\alpha (\partial_t + T_b \partial_\alpha) \bar R] \|_{H^{s-\frac12}} \lesssim 
\|R_\alpha\|_{H^{s-\frac12}}
\| (\partial_t + T_b \partial_\alpha) \bar R\|_{H^{s-\frac12}}
\]
where $\|R_\alpha\|_{H^{s-\frac12}} \lesssim B$,
and for the second factor we need the intermediate step
\begin{equation}\label{para-diff-R}
\| (\partial_t + T_b \partial_\alpha)  R\|_{H^{s-\frac12}}\lesssim_A A^2.
\end{equation}
Here we can replace $T_b$ with $b$ by estimating the error by a variation  of Lemma~\ref{l:para-err},
\[
\| (b-T_b) \partial_\alpha R \|_{H^{s-\frac12}}
\lesssim \|b\|_{H^{s-\frac12,\frac12}} \| R\|_{H^{s-\frac12,\frac12}} \lesssim_A A^2,
\]
where we have used the bound \eqref{b-A} for $b$.
For the material derivative of $R$ we can directly use the second equation in \eqref{ww2d-diff-real}, together with the $H^{s-\frac12}$ bounds for $\W$ and for $a$ (see \eqref{a:A}); this concludes the proof of \eqref{para-diff-R}.

A similar strategy applies for the last term in $D_t a$, where we can use the bound~\eqref{a:B} to write
\[
\| (b-T_b) \partial_\alpha a \|_{H^{s-\frac12}}
\lesssim \|b\|_{H^{s-\frac12,\frac12}} \| a\|_{H^{s}} \lesssim_A B.
\]
It remains to consider the two commutator terms in $D_t a$. Here we use Littlewood-Paley decompositions in the $\alpha$ direction in order 
to take advantage of the frequency ordering imposed 
by the paraproducts and by the projector $P$.
For the first commutator term we have the following frequency ordering
\[
P[ R  T_{b_\alpha} \bar R_\alpha] = 
\sum_{\nu < \mu \leq \lambda}
P[ b_{\nu,\alpha} \bar R_{\mu,\alpha} R_\lambda].
\]
This allows us to estimate, using almost orthogonality and \eqref{b-A},
\[
\begin{aligned}
\|P[ R  T_{b_\alpha} \bar R_\alpha] \|_{H^{s-\frac12}}^2 \lesssim  & \ \sum_{\nu < \mu \leq \lambda}
\| b_{\nu,\alpha}\|^2_{L^\infty} \|\bar R_{\mu,\alpha}\|_{L^\infty}^2 \lambda^{2s-1}\|R_\lambda\|_{L^2}^2 
\\
\lesssim & \  
\|b\|_{H^{s-\frac12,\frac12}}^2\|\bar R\|_{H^{s,\frac12}}^2
\|R\|_{H^{s-\frac12,\frac12}}^2 \sup_{\lambda}
\sum_{\nu < \mu \leq \lambda} \frac{\nu}{\lambda}
\\ 
\lesssim_A & \ A^4 B^2.
\end{aligned}
\]
Here it was important that we always have a derivative on the lowest frequency factor.

For the second commutator term we similarly denote by $\nu$, $\mu$ and $\lambda$ the dyadic frequencies of $b$, $R$ and $\bar R$, noting that 
all terms vanish unless $\nu,\mu \leq \lambda$.
We consider three cases:

a) $\mu = \nu = \lambda$, where we easily estimate 
each term separately. 

b) $\mu < \lambda$, in which case full cancellation occurs unless $\mu \leq \nu \leq \lambda$.

c) $\nu < \mu \approx \lambda$, in which case 
we need to estimate the expression
\[
(b_\nu - T_{b_\nu}) \partial_\alpha ( R_\lambda \bar R_{\mu,\alpha}).
\]
The first two cases are similar to the previous estimate, so we consider the third one. There 
the output vanishes unless the product 
$R_\lambda \bar R_{\mu,\alpha}$ is at frequency $\lesssim \nu$, therefore we can estimate
\[
\| (b_\nu - T_{b_\nu}) \partial_\alpha ( R_\lambda \bar R_{\mu,\alpha})\|_{H^{s-\frac12}}^2
\lesssim \sum_{\nu < \mu \approx \lambda}
\lambda^{2s-1} \nu^2 \| b_\nu\|_{L^\infty} 
\| \bar R_\mu\|_{L^\infty} \|R_\lambda\|_{L^2}^2
\]
after which this is similar to the earlier computation.

\end{proof}

The following lemma is also needed in our quest to show the RHS of the first equation in \eqref{ww2d-diff-real} plays a perturbative role.
\begin{lemma} \label{lem:M} Assume $s > (d+1)/2$. For the  auxiliary function $M$ we have:
\begin{itemize}
\item[(i)] The control norms allow for a bound as follows
\begin{equation}
\| M\|_{H^{s-\frac{1}{2}}}\lesssim AB.
\end{equation}

\item[(ii)] An energy bound also holds true
\begin{equation}
\left\|M\right\|_{H^k} \lesssim B(E^k)^{\frac{1}{2}}.
\end{equation}

\end{itemize}
\end{lemma}

\begin{proof}
Note that (i) is a special case of (ii) with $k = s-\frac12$.
Recall the definition of $M$
\[
M= \frac{R_\alpha}{1+\bar{\W}} + \frac{\bar R_\alpha}{1+\W} - b_\alpha  = \bar P[\bar RY_\alpha - R_\alpha \bar Y] + P[R\bar Y_\alpha - \bar R_\alpha Y].
\]
We are going to show that the  first term in $M$ can be bounded by the right-hand side of the inequality and the proof for the other three terms follows similarly. We will show that
\[
\| \bar P[\bar R Y_\alpha]\|_{H^k} 
\lesssim \| R\|_{H^{k,\frac12}} \Vert Y\Vert_{H^s} +\| R\|_{H^{s,\frac{1}{2}}} \Vert Y\Vert_{H^k}.
\]
This is a convolution estimate in the Fourier space, so it suffices to consider functions with positive Fourier transform. Given that $R$ and $Y$ are holomorphic, the projector $\bar P$ guarantees that the $\alpha$ frequency of $Y$ must be smaller than the $\alpha$ frequency of $\bar R$.
Then we can write 
\[
\begin{aligned}
\| \langle D\rangle^k \bar P[\bar R Y_\alpha]\|_{L^2}
\lesssim & \ \| \langle D\rangle ^{k} |D_\alpha|^\frac12 R 
 |D_\alpha|^\frac12 Y \|_{L^2}
+ \| |D_\alpha| R 
\langle D\rangle ^k Y \|_{L^2} \\
\lesssim & \  \|\langle D\rangle ^{k} |D_\alpha|^\frac12 R \|_{L^2}
\||D_\alpha|^\frac12 Y \|_{L^\infty}
+ \| |D_\alpha| R\|_{L^\infty}
\| \langle D\rangle ^k Y \|_{L^2}, 
\end{aligned}
\]
after which it suffices to use Sobolev embeddings for $R$ and $Y$.

\end{proof}

\section{Construction of regular solutions}
\label{s:existence}

We recall the equations 
\begin{equation} \label{ww2d-diff-re}
\left\{
\begin{aligned}
 & \W_{ t} + b \W_{ \alpha} + \frac{(1+\W) R_\alpha}{1+\bar \W}   =  (1+\W)M,
\\
& R_t + bR_\alpha  = i(1+a)\frac{\W}{1+\W} -i a.
\end{aligned}
\right.
\end{equation}
Here  we devise an iterative scheme 
to solve them in higher regularity spaces.
Precisely we will prove the following 

\begin{theorem}
\label{t:existence}
Let $k > \frac{d+1}2$  be an integer. Then for any initial data $(\W_0,R_0) \in \H^k$ for 
the differentiated gravity wave system \eqref{ww2d-diff-re} there exists a local solution $(\W, R)$ in $C([0,T];\H^k)$ with $T$ depending on the size of the initial data.     
\end{theorem}

\begin{proof}
Uniqueness follows from Theorem~\ref{thm:diff}.
    To construct the solutions, we iteratively define a sequence of approximate solutions $(\W^m,R^m)$ by
\begin{equation} 
\left\{
\begin{aligned}
 & \W^{m+1}_{ t} + P^\sharp\left[ b^m \W^{m+1}_{ \alpha} + \frac{(1+\W^m) R^{m+1}_\alpha}{1+\bar \W^m} \right]  =  P^\sharp\left[(1+\W^m)M^m\right],
\\
& R^{m+1}_t + P^\sharp\left[b^mR^{m+1}_\alpha  - i(1+a^m)T_{(1-Y^m)^2} \W^{m+1}\right] = P^\sharp\left[i(1+a^m)( Y^m - T_{(1-Y^m)^2} \W^{m})  - i a^m\right].
\end{aligned}
\right.
\end{equation}
where the paraproduct is defined only relative to the $\alpha$ direction. 

Here we start with $(\W^0,R^0) = (0,0)$, and for $m \geq 1$ we set the initial data as
\[
(\W^{m}_0, R^m_0) =(\W_0, R_0). 
\]

We denote the initial data size by $K$, where
\[
K := \|(\W_0,R_0)\|_{\H^k}.
\]
Now our goal is to inductively prove 
\begin{enumerate}[label=(\roman*)]
\item Uniform bounds
\begin{equation}\label{ee-k}
\|   (\W^m,R^m)\|_{L^\infty([0,T];\H^k)} \lesssim C K  
\end{equation}
with a large universal constant $C$, on a short  time interval $[0,T]$ depending only on $K$.

\item Difference bounds
\begin{equation}\label{ediff-k}
\| (\W^{m+1},R^{m+1}) - (\W^m,R^m)\|_{L^\infty([0,T];\H^0)} \lesssim c  \|(\W^{m},R^{m}) - (\W^{m-1},R^{m-1})\|_{L^\infty([0,T];\H^0)}
\end{equation}
\noindent with a small constant $c < 1$.
\end{enumerate}
Assuming these properties hold, it follows that the sequence $(\W^m,R^m)$ converges in $L^\infty \H^0$  and is bounded in $L^\infty \H^k$. We denote the limit by $(\W,R)$, which is 
in $L^\infty H^k$. Further, by interpolation we have 
\[
(\W^m,R^m) \to (\W,R) \quad \text{in } L^\infty H^s, \qquad s < k.
\]
This suffices in order to guarantee that we can pass to the limit in the iteration scheme and conclude that $(\W,R)$ is a solution to \eqref{ww2d-diff-re} with initial data $(\W_0,R_0)$.

\medskip

It remains to prove the bounds \eqref{ee-k}
and \eqref{ediff-k}.

\medskip

\emph{ Proof of the uniform bound \eqref{ee-k}.}
This is similar to the proof of the energy estimates in Theorem~\ref{t:ee}. For a multiindex $\kappa$ of length $k$ we 
differentiate the equation for $(\W^{m+1},R^{m+1})$ $k$ times. 
The differentiated variables $(\W^{m+1,[\kappa]},R^{m+1,[\kappa]})$ solve a linear 
system of the form
\begin{equation}\label{diff-m-kappa} 
\left\{
\begin{aligned}
 & \W^{m+1,[\kappa]}_{ t} + P^\sharp\left[b^m \W^{m+1,[\kappa]}_{ \alpha} + \frac{(1+\W^m) R^{m+1,[\kappa]}_\alpha}{1+\bar \W^m} \right]  =  F^{m,[\kappa]},
\\
& R^{m+1,[\kappa]}_t + P^\sharp\left[b^mR^{m+1,[\kappa]}_\alpha  - i(1+a^m)T_{(1-Y^m)^2} \W^{m+1,[\kappa]}\right] = G^{m,[\kappa]},
\end{aligned}
\right.
\end{equation} 
where we claim that the source terms $(F^{m,[\kappa]}, G^{m,[\kappa]})$ satisfy the bound
\begin{equation}\label{FG-m-kappa}
\| (F^{m,[\kappa]}, G^{m,[\kappa]})\|_{\H^0}
\lesssim_{A_m,A_{m+1}} (B_m+B_{m+1}) (\| (\W^m,R^m)\|_{\H^k}+  \| (\W^{m+1},R^{m+1})\|_{\H^k}).
\end{equation}
The proof of this bound is almost identical 
to the proof of \eqref{FG-kappa}, with the only differences stemming from our use of the paradifferential decomposition. Precisely, from the left-hand side of the second equation we 
need to estimate the term 
\[
\| (1+a^m) ((1-Y^m)^2 - T_{(1-Y^m)^2}) \W^{m+1,[\kappa]}\|_{H^{0,\frac12}} \lesssim_{A_m} B_m \|\W^{m+1,[\kappa]}\|_{L^2},
\]
for which we can use Lemma~\ref{l:para-err}.
From the right-hand side of the second equation
we 
need to estimate the term 
\[
\| (1+a^m) (Y^{m, [\kappa]}  - T_{(1-Y^m)^2}) \W^{m,[\kappa]}\|_{H^{0,\frac12}} \lesssim_{A_m} B_m \|\W^{m}\|_{H^k}. 
\]
Applying the chain rule for the first term and peeling off good terms with distributed derivatives we are left with 
\[
\| (1+a^m) ((1-Y^m)^2  - T_{(1-Y^m)^2}) \W^{m,[\kappa]}\|_{H^{0,\frac12}} \lesssim_{A_m} B_m \|\W^{m}\|_{H^k},
\]
which is identical to the previous bound.

Once we have the bound \eqref{FG-m-kappa}, we can apply Theorem~\ref{plin-reduced} for the equation \eqref{diff-m-kappa} to obtain
\[
\frac{d}{dt} E_{lin}(\W^{m+1,[\kappa]},R^{m+1,[\kappa]})  
\lesssim_{A_m,A_{m+1}} (B_m + B_{m+1}) 
(\| (\W^m,R^m)\|_{\H^k}+  \| (\W^{m+1},R^{m+1})\|_{\H^k} ).
\]
Summing up over $\kappa$ with $|\kappa|\leq k$
and applying Gronwall's inequality we arrive at 
\[
\| (\W^{m+1},R^{m+1})(t)\|_{\H^k} \lesssim e^{C(A_m,A_{m+1})(B_m+B_{m+1}) t} (\| (\W_0,R_0)\|_{\H^k}  + t \| (\W^{m},R^{m})\|_{L^\infty(0,T;\H^k)} ).
\]
Assuming $T$ is small enough, $T \ll_K 1$,
this allows us to conclude the inductive proof of \eqref{ee-k}.

\medskip

\emph{ Proof of the difference bound \eqref{ediff-k}.} This is similar to the proof of the energy estimates for the linearized equation in Theorem~\ref{plin-short}. Subtracting the equations for $(\W^{m},R^{m})$ and $(\W^{m+1},R^{m+1})$ we obtain an equation for the difference, which we write as 
\begin{equation}\label{diff-m} 
\left\{
\begin{aligned}
 & (\W^{m+1}-W^m)_{ t} + P^\sharp\left[ b^m (\W^{m+1}-\W^m)_{ \alpha} + \frac{(1+\W^m) (R^{m+1}-R^m)_\alpha}{1+\bar \W^m} \right]  =  F^{m},
\\
& (R^{m+1}-R^m)_t +  P^\sharp\left[b^m (R^{m+1}-R^m)_\alpha  - (1+a^m)T_{(1-Y^m)^2 }(\W^{m+1}-\W^m)\right] = G^{m}.
\end{aligned}
\right.
\end{equation}
Here we claim that the source terms $(F^{m}, G^{m})$ satisfy the bound
\begin{equation}\label{FG-m}
\begin{aligned}
\| (F^{m}, G^{m})\|_{\H^0}
\lesssim_{A_m,A_{m+1}} & (B_m+B_{m+1}) (\| (\W^{m+1}-\W^m,R^{m+1}-R^m)\|_{\H^0}\\
&\hspace*{.3cm}+  \| (\W^m-\W^{m-1},R^m-R^{m-1})\|_{\H^0}).
\end{aligned}
\end{equation}
This is proved in the same manner as Theorem \ref{thm:diff}, with minor differences arising, as above, from our use of paraproducts
in our iterations. Details are omitted for brevity.
\end{proof}

\section{Rough solutions as limits of smooth solutions}

\label{s:rough sol}

 Our aim here is to prove our main local well-posedness result for quasiperiodic solutions 
for the differentiated water wave system in Theorem~\ref{thm:lwp}. The rough solutions  will be obtained as limits of the regular solutions in Theorem~\ref{t:existence},
using the $\H^k$ energy estimates in Theorem~\ref{t:ee} and the difference bounds in Theorem~\ref{thm:diff}. 

To begin with, consider holomorphic zero average initial data $(\W_0,R_0) \in H^s$, of size $K$, and so that 
\begin{equation}\label{nocorner}
  \| Y_0\|_{L^\infty} \leq C.
\end{equation}

We approximate this initial data with regularized initial data 
\[
(\W_0^m,R_0^m) = P_{\leq m} (\W_0,R_0), \qquad m \geq m_0.
\]
Here $m_0$ is taken large enough, depending on $K$ and $C$, so that the condition \eqref{nocorner} holds uniformly in $m$. 

These regularized data are smooth, so 
by Theorem~\ref{t:existence}, local solutions 
$(\W^m,R^m)$ exist, though apriori only on time intervals which depend on $m$. We aim to show that these solutions extend uniformly to a time interval $[0,T]$ depending only on $K$ and $C$,
and with uniform bounds in time. We will achieve this using a bootstrap argument.

 To set up the bootstrap we use the language 
 of frequency envelopes, following the procedure
 described in \cite{IT-primer}. To begin with, we place the initial data under a slowly varying $\ell^2$ normalized frequency envelope, 
 \[
\| P_m (\W_0,R_0)\|_{\H^0} \lesssim 2^{-ms} c_m K,
\]
so that 
\[
\|(\W_0,R_0)\|_{\H^s}^2 \approx \sum c_k^2 K^2 \approx K^2.
\]
Then for the sequence of initial data  we have 
bounds as follows:

\begin{enumerate}[label=$(\roman*)_0$]
    \item Uniform $H^s$ bound,
  \begin{equation}
  \| (\W_0^m,R_0^m)\|_{\H^s} \lesssim K.    
  \end{equation}  
\item Higher regularity,
 \begin{equation}
  \| (\W_0^m,R_0^m)\|_{\H^k} \lesssim  2^{(k-s)m}  c_m  K.
  \end{equation} 
\item Difference bound,
 \begin{equation}
  \| (\W_0^{m+1}-\W_0^m,R_0^{m+1} - R_0^m)\|_{\H^0} \lesssim  2^{-sm}  c_m  K.
 \end{equation} 
\end{enumerate}

We claim that these bounds transfer to the 
corresponding solutions on a time interval $[0,T]$ which depends only on $K$ and $C$,

\begin{enumerate}[label=$(\roman*)$]
    \item Uniform $H^s$ bound,
  \begin{equation}
  \| (\W^m,R^m)\|_{L^\infty[0,T;\H^s]} \lesssim K.     
  \end{equation}  
\item Higher regularity,
 \begin{equation}
  \| (\W^m,R^m)\|_{L^\infty[0,T;\H^k]} \lesssim  2^{(k-s)m}  c_m  K.
  \end{equation} 
\item Difference bound,
 \begin{equation}
  \| (\W^{m+1}-\W^m,R^{m+1} - R^m)\|_{L^\infty[0,T;\H^0]} \lesssim  2^{-sm}  c_m  K.
 \end{equation} 
\end{enumerate}

 We first prove these properties assuming 
 that the solutions exist uniformly on $[0,T]$
 and satisfy a bootstrap assumption
\begin{equation}\label{boot}
  \| (\W^m,R^m)\|_{L^\infty[0,T;\H^s]} \lesssim C_0 K     
  \end{equation} 
with a large universal constant $C_0$. The key point is that we want to prove (i)-(ii)-(iii)
with implicit constants which do not depend on $C_0$. To accomplish this, we will instead 
allow $T$ to depend on $C_0$. We also remark 
that in order to bootstrap only finitely many quantities, we can restrict the argument to a finite range of regularization scales $m_0 < m < m_1$, but where $m_1$ is allowed to be arbitrarily large. At the conclusion of the bootstrap argument we let $m_1 \to \infty$ 
to cover the full range.

The bootstrap assumption guarantees that 
the control parameters $A_m$ and $B_m$ 
satisfy the uniform bounds
\[
A_m, B_m \lesssim C_0 K.
\]
By Theorem~\ref{t:ee} and Gronwall's inequality 
we then obtain the bound 
\[
 \| (\W^m,R^m)(t)\|_{\H^k} \leq e^{t C(C_0 K)} 
 \| (\W_0^m,R_0^m)\|_{\H^k} \lesssim e^{t C(C_0 K)} 2^{(k-s)m}  c_m  K.
\]
Choosing $T$ small enough so that $T C(C_0 K) \leq 1$, we arrive at the bound in (ii). 

A similar argument but using instead Theorem~\ref{thm:diff} yields the difference bound in (iii). Finally to prove (ii)
we express $(\W^m,R^m)$ as a telescopic sum 
\[
(\W^m,R^m) = (\W^{m_0},R^{m_0}) + \sum_{n=m_0}^{m-1}  (\W^{n+1}-\W^n,R^{n+1} - R^n)
\]
using the bounds in (ii) and (iii) to estimate the differences in a higher and lower norm, in order to gain almost orthogonality of the summands.  For later use we note that the same argument also yields dyadic bounds
\begin{equation}\label{fe-m}
  \| P_n (\W^m,R^m)\|_{L^\infty[0,T;\H^s]} \lesssim c_n K,    
\end{equation}
which are useful for $n \leq m$, otherwise they are 
superseded by (ii). Another consequence 
is the difference bound 
\begin{equation}\label{diff-Hs}
  \| (\W^{m_2},R^{m_2}) - (\W^m,R^m)\|_{L^\infty[0,T;\H^s]} \lesssim c_{[m,m_2]} K, \qquad  c_{[m,m_2]}^2   = \sum_{n=m}^{m_2} c_n^2. 
\end{equation}

To complete the bootstrap we use a continuity 
argument.  Let $T_1 \leq T$ be maximal so that 
\eqref{boot} holds in $[0,T_1]$ for $m_0 \leq m \leq m_1$. If $T_1 = T$ then (i)-(ii)-(iii) hold in $[0,T]$ by the bootstrap argument above,
and we are done.

Otherwise, by the bootstrap argument, 
(i)-(ii)-(iii) hold in $[0,T_1]$. In particular 
(i) holds at $T_1$. Then Theorem~\ref{t:existence} shows that we can extend the solutions $(\W^m,R^m)$ as regular solutions beyond $T_1$, and by (i), \eqref{boot}
also holds beyond $T_1$. This contradicts the maximality of $T_1$. 

Now we consider the limit of $(\W^m,R^m)$ 
as $m \to \infty$. By \eqref{diff-Hs}, this limit exists in $L^\infty[0,T; \H^s]$. We denote it 
by $(\W,R)$ which as a uniform limit of continuous functions is also continuous in time with values in $\H^s$. Passing to the limit in the equations we also easily see that $(\W,R)$
solves \eqref{ww2d-diff-real}.
Finally, from \eqref{fe-m} we also obtain
\begin{equation}\label{fe-full}
  \| P_n (\W,R)\|_{L^\infty[0,T;\H^s]} \lesssim c_n K,       
\end{equation}
which shows that the frequency envelope bounds for the solutions are carried over from the initial data.

The remaining component of the proof is to show continuous dependence of the solution with respect to the initial data in the strong topology. This argument follows \cite{IT-primer} fully, and is omitted.

\section*{Acknowledgments}
The authors thank the workshop ``Nonlinear Water Waves: Rigorous Analysis and
Scientific Computing'' at Banff International Research Station, where this collaboration was initiated. M.I. gratefully acknowledges support from  the NSF grant DMS-2348908, from a Miller Visiting Professorship at UC Berkeley during the Fall semester of 2023,  from the Simons Foundation through a Simons Fellowship in the Spring semester of 2024, and from a Vilas Associate Fellowship.  J.W. gratefully acknowledges support from the U.S. Department of Energy, Office of Science, Office of Advanced Scientific Computing Research's Applied Mathematics Competitive Portfolios program under Contract No. AC02-05CH11231.  X.Z. gratefully acknowledges support from the National Science Foundation through grant DMS-2511663.

\bibliography{wwbib}

\bibliographystyle{plain}

\end{document}